\documentclass[12pt,a4paper,reqno]{amsart}
\usepackage{IEEEtrantools}

\title[Compositional Construction of Approximate
Abstractions]{Compositional Construction of Approximate Abstractions of
Interconnected Control Systems}

\author{Matthias Rungger}
\address{Department of Electrical and Computer Engineering at the Technical University of Munich, 80333 Munich, Germany.} 
\email{matthias.rungger@tum.de}

\author{Majid Zamani}
\address{Department of Electrical and Computer Engineering at the Technical University of Munich, 80333 Munich, Germany.} 
\email{zamani@tum.de}

\date{}

\usepackage[linewidth=.5pt]{mdframed} 

\usepackage{amsmath}
\usepackage{amsthm}
\usepackage{amssymb}

\usepackage[usenames,dvipsnames,svgnames,table]{xcolor}

\usepackage{ltl}
\usepackage{tikz}
\usetikzlibrary{calc,shapes,arrows}

\newtheorem{theorem}{Theorem}
\newtheorem{lemma}{Lemma}
\newtheorem{corollary}{Corollary}
\newtheorem{definition}{Definition}
\newtheorem{remark}{Remark}

\newcommand{\intcc}[1]{\ensuremath{{\left[#1\right]}}}
\newcommand{\intoc}[1]{\ensuremath{{\left]#1\right]}}}
\newcommand{\intco}[1]{\ensuremath{{\left[#1\right[}}}
\newcommand{\intoo}[1]{\ensuremath{{\left]#1\right[}}}

\newcommand{\B}{\mathbb{B}}
\newcommand{\R}{\mathbb{R}}
\newcommand{\N}{\mathbb{N}}

\DeclareMathOperator{\im}{im}

\usepackage{ltl}
\newcommand{\G}{\LTLsquare}

\newcommand{\F}{\LTLdiamond}

\newcommand{\id}{{\mathrm{id}}}

\begin{document}

\maketitle        
         
\begin{abstract}                          
We consider a compositional construction of approximate abstractions of
interconnected control systems. In our framework, an abstraction acts as a
substitute in the controller design process and is itself a continuous control
system. The abstraction is related to the concrete control system via a
so-called simulation function: a Lyapunov-like function, which is used to
establish a quantitative bound between the behavior of the approximate
abstraction and the concrete system. In the first part of the paper,  we provide
a small gain type condition that facilitates the compositional construction of
an abstraction of an interconnected control system together with a simulation
function from the abstractions and simulation functions of the individual
subsystems. In the second part of the paper, we restrict our attention to linear
control system and characterize simulation functions in terms of controlled
invariant, externally stabilizable subspaces. Based on those characterizations,
we propose a particular scheme to construct abstractions for linear control
systems. We illustrate the compositional construction of an
abstraction on an interconnected system consisting of four linear subsystems. We
use the abstraction as a substitute to synthesize a controller to enforce a certain linear temporal logic specification.
\end{abstract}

\section{Introduction}

One way to address the inherent difficulty in modeling, analyzing and
controlling complex, large-scale, interconnected systems, is to apply a divide-and-conquer
scheme~\cite{Keating11}. In this approach, as a first step, the overall system is partitioned
in a number of reasonably sized components, i.e., subsystems.
Simultaneously, a number of appropriate interfaces to connect the
individual subsystems are introduced. Subsequently, the analysis and the design of
the overall system is reduced to those of the subsystems. There exist
different reasoning schemes to ensure the correctness of such a
component-based, compositional analysis and design procedure. One scheme, which is often
invoked in the formal methods community, is called \emph{assume-guarantee
reasoning}, see
e.g.~\cite{MisraChandy81,HenzingerQadeerRajamaniTasiran02,Frehse05}.
Here, one establishes the correctness of the composed system by
guaranteeing that each subsystem is correct, i.e., satisfies its
specification, under the assumption that all other subsystems are
correct. The assume-guarantee reasoning is always correct, if there is
no circularity between assumptions and guarantees. In the case of circular
reasoning, some additional ``assume/guarantee'' assumptions
are imposed. Another approach, which is known from control theory,
invokes a so
called \emph{small gain} condition, see
e.g.~\cite{JiangTeelPraly94,DullerudPaganini00,DIW11,DashkovskiyRuefferWirth10} to establish the
stability of the interconnected system. For example
in~\cite{DIW11,DashkovskiyRuefferWirth10},
the authors assume that the \emph{gain functions} that are associated
with the Lyapunov functions of the individual subsystems satisfy
a certain ``small gain'' condition. The condition certifies a
small (or weak) interaction of the subsystems, which
prevents an amplification of the signals across possible feedback
interconnections. Similarly to the assume-guarantee reasoning, the small
gain condition is always satisfied in the absence of any feedback
interconnection \cite[and references therein]{DashkovskiyRuefferWirth10}.

In this paper, we use the later reasoning and present
a method for the compositional construction of approximate abstractions of
interconnected nonlinear control systems. In our
approach, an abstraction is itself a continuous control system
(possibly with lower dimension), which
is used as a substitute in the controller design process. The
correctness reasoning from the abstraction to the concrete system is
based on a notion of
\emph{simulation function}, which relates the concrete system with
its abstraction. Simulation functions provide a quantitative bound between the behavior of
the concrete systems and their abstractions. We employ a small gain
type condition to construct a simulation function that relates the abstract
interconnected system to the concrete interconnected system from the
simulation functions of the individual subsystems. In the second part
of the paper, we focus on the construction of abstractions (together
with the associated simulation functions) of linear
control systems.
First, we characterize simulation functions in terms
of controlled invariant, externally stabilizable subspaces. Subsequently,
we propose a particular construction of abstractions of linear control
systems. We conclude the paper with the construction of an abstraction
together with a simulation function of an interconnected system
consisting of four linear subsystems. We use the constructed
abstraction as a substitute in the controller synthesis procedure
to enforce a certain linear temporal logic property \cite{BK08} on the
concrete interconnected system. As we demonstrate, the controller synthesis would not have been possible without
the use of the abstraction.

{\bf Related Work.} 
Compositional reasoning schemes  for
verification in connection with
abstractions of control systems are developed
in~\cite{TabuadaPappasLima04,Frehse05,KvdS10}. The methods employ
exact notions of abstractions which are based on simulation relations~\cite{Frehse05,KvdS10} and simulation
maps~\cite{TabuadaPappasLima04}, for which constructive procedures exist only for rather restricted classes of control systems, e.g. linear control
systems~\cite{vdS04} and linear hybrid automata~\cite{Frehse05}. In
contrast to the exact notions, the approximate abstractions which we
study in this paper are based on simulation functions whose structures are
closely related to (incremental) Lyapunov functions. Thus, advanced
nonlinear control techniques developed to construct Lyapunov functions
have the potential to also be used to construct
simulation functions. 
For example the toolbox developed  in~\cite{MurthyIslamSmolkaGrosu15} uses
sum-of-squares techniques to construct bisimulation functions to
relate nonlinear control systems.

An early approach to the compositional construction of simulation
functions is given in~\cite{Gir07}, where the 
interconnection of two subsystems is studied. 
Compositional schemes for general
interconnected systems for the
construction of finite abstractions of linear and nonlinear control systems
are presented in~\cite{TI08} and~\cite{PPdB14}, respectively. Like in this paper, small gain type
conditions are used to facilitate the compositional construction.
As in our framework an
abstraction is itself a continuous control system (potentially with lower dimension), the benefits of the proposed
scheme are not limited to synthesis procedures based on finite abstractions, and
therefore are potentially useful for a great variety of controller synthesis
schemes, most notably computationally expensive schemes (in terms of the state
space dimension of the system) such as~\cite{BM08,
BemporadMorariDuaPistikopoulos02,YTCBB12,RMT13}.
Nevertheless, as we demonstrate by an example, even for a
synthesis scheme based on finite abstractions, we can apply our
results as a first pre-processing step to
reduce the dimensionality of a given control system, before the construction of
the finite abstraction, and therefore substantially reduce the computational
complexity.

As we seek abstractions with reduced state space dimensions, our
approach is closely related to the rich theory of model order
reduction~\cite{Antoulas05}. Specifically, the construction of abstractions of
linear control systems (similar to the Krylov subspace methods and balanced order reduction
schemes) can be classified as projection based
methods~\cite{VillemagneSkelton87}. Additionally, similar to~\cite{SM09}, the
proposed compositional construction
of abstractions of interconnected control systems leads to a structure preserving
reduction technique. 
While in~\cite{Antoulas05,VillemagneSkelton87,SM09} the model mismatch is established
with respect to $\mathcal{H}_2/\mathcal{H}_\infty$ norms, we use simulation functions
to derive $\mathcal{L}_\infty$ error bounds, which are essential to reason about
complex properties, e.g. linear temporal logic properties \cite{BK08}, across related systems.

To summarize, our contribution
is twofold: 1) We present a small gain type condition to
construct an abstraction of an interconnected system and a corresponding simulation function from the abstractions of the subsystems and their simulation functions. It is neither limited to
two interconnected systems~\cite{Gir07}, nor to synthesis schemes based on finite
abstractions~\cite{TI08,PPdB14}.
2) We characterize
simulation functions for linear subsystems in terms of controlled invariant,
externally stabilizable subspaces, which leads to constructive procedures to
determine abstractions of linear systems. Simulation functions for linear systems
have been used in~\cite{GP09,TH12,FuShahTanner13}. However, a geometric characterization
of simulation functions, similar to~\cite{vdS04}, was missing. Moreover, this
characterization allows to show that the conditions proposed in~\cite{GP09} to
construct abstractions are not only sufficient, but actually also necessary.

A preliminary version of this work appeared
in~\cite{RZ15}. In this paper we present a less restrictive small
gain condition and provide a novel geometric
characterization of simulation functions for linear control systems.

\section{Notation and Preliminaries}
We denote by $\N$ the set of non-negative integers and by
$\R$ the set of real numbers.
We annotate those symbols with subscripts to restrict those sets in
the obvious way, e.g. $\R_{>0}$ denotes the positive real numbers.
We use $\R^{n\times m}$, with $n,m\in\N_{\ge1}$,
to denote the vector space of real matrices with $n$ rows and $m$ columns. The identity matrix in $\R^{n\times n}$ is denoted by
$I_n$. For $a,b\in\R$ with $a\le b$, we denote the closed, open and half-open intervals in $\R$ by $\intcc{a,b}$,
$\intoo{a,b}$, $\intco{a,b}$, and $\intoc{a,b}$, respectively. For $a,b\in\N$ and $a\le b$, we
use $\intcc{a;b}$, $\intoo{a;b}$, $\intco{a;b}$, and $\intoc{a;b}$ to
denote the corresponding intervals in $\N$.
Given $N\in\N_{\ge1}$, vectors $x_i\in\R^{n_i}$, $n_i\in\N_{\ge1}$ and $i\in\intcc{1;N}$, we
use $x=(x_1;\ldots;x_N)$ to denote the vector in $\R^N$ with
$N=\sum_i n_i$ consisting of the concatenation of vectors~$x_i$.

We use $|\cdot |$ to denote the Euclidean norm of vectors in 
$\R^n$ as well as the spectral
norm, of matrices in $\R^{n\times m}$. 
Also for
$\xi:\R_{\ge1}\to \R^n$ we introduce $||\xi||_\infty:=\sup_{t\in\R_{\ge0}} |\xi(t)|$.

Given a function $f:\R^n\to \R^m$ and $\bar x\in\R^m$, we use
$f\equiv \bar x$ to denote that $f(x)=\bar x$ for all $x\in\R^n$. If
$\bar x$ is the zero vector, we simply write $f\equiv 0$.  The identity
function in $\R^n$ is denoted by $\id$, where the
dimension is always clear from the context. We use
$\mathsf{D}V:\R^n\to\R^{1\times n}$ to denote the gradient of a scalar
function $V:\R^n\to\R_{\ge0}$ and
$\mathsf{D}^{+}V(x,v)=\limsup_{t\to 0,t>0}\tfrac{1}{t}(V(x+tv)-V(x))$ to
denote the upper-right Dini derivative in the direction of $v$. 
Given two subsets $A,B\subseteq\R^n$,
we use $A+B=\{a+b\mid a\in A,b\in B\}$ to denote the Minkowsky set
addition.

We use the usual notation $\mathcal{K}$, $\mathcal{K}_\infty$ and $\mathcal{KL}$
to denote the different classes of comparison functions, see
e.g.~\cite{DashkovskiyRuefferWirth10}. Moreover, we use MAF$_n$ to denote the
set of \emph{monotone aggregation functions}~\cite{DashkovskiyRuefferWirth10}, i.e., the class of functions
$\mu:\R^n_{\ge0}\to \R_{\ge0}$ that satisfy: i) $\mu(s)\ge 0$ for all
$s\in\R_{\ge0}^n$ and $\mu(s)=0$ iff $s= 0$; ii) for $s,r\in\R^n_{\ge0}$ 
$s_i>r_i$ for all $i\in\intcc{1;n}$ implies $\mu(s)>\mu(r)$; iii) $|s|\to
\infty$ implies $\mu(s)\to \infty$.

We recall some concepts from the geometric approach to linear systems
theory~\cite{BM92}. Let $A\in\R^{n\times n}$ and $B\in\R^{n\times m}$. We use the usual
symbols $\im B$ and $\ker B$ to denote image and kernel of $B$.
A linear subspace $S\subseteq \R^n$ is called
$(A,B)$-\emph{controlled invariant} if 
there exists a matrix $K$ (of appropriate
dimension) such that $(A+BK)S \subseteq S$, where 
the matrix-subspace product is given by $AS:=\{x\in \R^n\mid \exists_{y\in S}\,x=Ay\}$.
An $(A,B)$-controlled invariant subspace $S\subseteq \R^n$ is 
$(A,B)$-\emph{externally stabilizable} if there exists a matrix $K$ (of appropriate
dimension) such that $(A+BK) S\subseteq S$ and
$(A+BK)|_{\R^{n}/S}$ is Hurwitz, i.e., the real parts of all the
eigenvalues are strictly less than~$0$. Here, $(A+BK)|_{\R^n/S}$
denotes the map induced by $(A+BK)$ on the quotient space  $\R^n/S$,
see~\cite[Def.~3.2.2]{BM92}. 

\section{Background and Motivation}

In this work, we study nonlinear control systems of the following
form.
\begin{definition}\label{d:sys}
A \emph{control system} $\Sigma$ is a tuple
\begin{IEEEeqnarray}{c}\label{e:sys}
\Sigma=\left(X,U,W,\mathcal{U},\mathcal{W},f,Y,h\right),
\end{IEEEeqnarray}
where $X\subseteq \R^n$, $U\subseteq \R^m$, $W\subseteq \R^p$, and
$Y\subseteq \R^q$ are the \emph{state space},
\emph{external input space}, 
\emph{internal input space}, 
and \emph{output space}, respectively. We use the symbols $\mathcal{U}$ and $\mathcal{W}$ 
to, respectively, denote the set 
of piecewise
continuous functions
\mbox{$\nu:\R_{\ge0}\to U$} and $\omega:\R_{\ge0}\to W$.
The function
$f:X\times U\times W\to\R^n$ is the \emph{vector field} and
$h:X\to Y$ is the \emph{output function}.
\end{definition}

In our definition of a control system, we distinguish between
\emph{external} inputs $u\in U$ and \emph{internal} inputs
$w\in W$. The purpose of this
distinction will become apparent in Section~\ref{s:inter} where we introduce the
interconnection of systems. Basically, we use the internal inputs to define
the interconnection. For now, without referring to the
interconnection, we can interpret the internal inputs as
\emph{disturbances} over which we have no control and the
external inputs as \emph{control} inputs which we are allowed to
modify.

A control system $\Sigma$ induces a set of trajectories by the 
differential equation
\begin{IEEEeqnarray}{c}\label{e:ode}
  \begin{IEEEeqnarraybox}[\relax][c]{rCl}
    \dot\xi(t)&=&f(\xi(t),\nu(t),\omega(t)),\\
    \zeta(t)&=&h(\xi(t)).%
  \end{IEEEeqnarraybox}
\end{IEEEeqnarray}
A \emph{trajectory}
of $\Sigma$ is a tuple $(\xi,\zeta,\nu,\omega)$, consisting of a
\emph{state trajectory} $\xi:\R_{\ge0}\to X$, an \emph{output
trajectory} \mbox{$\zeta:\R_{\ge0}\to Y$}, and input trajectories
$\nu\in\mathcal{U}$ and $\omega\in\mathcal{W}$,
that satisfies~\eqref{e:ode} for almost all times
$t\in\R_{\ge0}$. We often use $\xi_{x,\nu,\omega}$ and
$\zeta_{x,\nu,\omega}$ to
denote the state trajectory and output trajectory associated with 
input trajectories $\nu\in\mathcal{U}$, $\omega\in\mathcal{W}$ and \emph{initial state}
$x=\xi(0)$, without explicitly referring to the tuple $(\xi,\zeta,\nu,\omega)$.

Throughout the paper, we impose the usual regularity
assumptions~\cite{LSW96} on $f$ and assume that $X$ is strongly invariant and $\Sigma$ is forward
complete, so that for every initial state
and input trajectories, there exists a unique state trajectory which is defined
on the whole semi-axis.

We recall the notion of simulation function, introduced
in~\cite{GP09}, which we adapt here to match our notion of control system
with internal and external inputs. As we show in Section~\ref{s:lin},
for the case of linear control systems, our notion of simulation function is
related to the notion of simulation relation used in~\cite{vdS04}.

\begin{definition}\label{d:sf} Let
$\Sigma=\left(X,U,W,\mathcal{U},\mathcal{W},f,Y,h\right)$ and
$\hat \Sigma=(\hat X,\hat U,\hat W,\hat{\mathcal{U}},\hat{\mathcal{W}},\hat
f,\hat Y,\hat h)$ be
two control systems with $p=\hat p$ and $q=\hat q$. A 
continuous 
function \mbox{$V:\hat X\times X\to\R_{\ge0}$}, locally Lipschitz on 
\mbox{$(\hat X\times X)\setminus V_0$ with $V_0=\{(\hat x,x)\mid V(\hat
x,x)=0\}$}, is
called a \emph{simulation function from $\hat\Sigma$ to $\Sigma$} if
for every
$x\in X$, $\hat x\in \hat X$, $\hat u\in \hat U$, $\hat w\in \hat W$, there
exists $u\in U$ so that for all $w\in W$ we have the following inequalities
\begin{IEEEeqnarray}{rCl}\label{e:bsf:1}
\alpha(|\hat h(\hat x)-h(x)|)&\le& V(\hat x, x),\\
\mathsf{D}^{+} V\left((\hat x, x),
\begin{bmatrix}
\hat f(\hat x,\hat u,\hat w)\\ \label{inequality1}
f(x,u,w)
\end{bmatrix}
\right)
&\le&
-\lambda(V(\hat x, x))\\\notag
&&\hspace{-2cm}+\rho(|\hat u|)+ \mu(|w_1-\hat w_1|,\ldots,|w_p-\hat
w_p|).
\end{IEEEeqnarray}
for some fixed $\alpha,\lambda\in \mathcal{K}_\infty$, $\rho\in\mathcal{K}\cup
  \{0\}$ and MAF$_p$ $\mu$.

\end{definition}

Let us point out some differences between our definition of simulation
function and Definition~1 in~\cite{GP09}. Here, for the sake of a
simpler presentation, we
simply assume that for every $x$, $\hat x$, $\hat u$, $\hat w$ there
exists a $u$ so that \eqref{inequality1} holds for all $w$. While
in~\cite{GP09} the authors use an \emph{interface function} $k$ to provide
the input $u=k(x,\hat x, \hat u,\hat w)$ that
enforces~\eqref{inequality1}.
Moreover, in Definition~1 in~\cite{GP09} there is no
distinction between internal and external inputs and, therefore,
\mbox{$\mu(|w_1-\hat w_1|,\ldots,|w_p-\hat w_p|)$} does not appear on the
right-hand-side of~\eqref{inequality1}.  
Furthermore, we formulate the decay
condition~\eqref{inequality1} in ``dissipative'' form \cite{DIW11}, while in
\cite[Def.~1]{GP09} the decay condition is formulated in
``implication'' form \cite{DIW11}.

The following theorem shows the importance of the existence of a
simulation function according to Definition~\ref{d:sf}.

\begin{theorem}\label{theorem1}
Consider 
$\Sigma=\left( X,U,W,\mathcal{U},\mathcal{W},f,Y,h\right)$
and 
$\hat \Sigma=(\hat X,\hat U,\hat
W,\hat{\mathcal{U}},\hat{\mathcal{W}},\hat f,\hat Y,\hat h)$ with
$q=\hat q$ and $\hat p=p$.
Suppose $V$ is a simulation function from $\hat\Sigma$ to $\Sigma$.
Then, there exist a $\mathcal{KL}$ function $\beta$ and
$\mathcal{K}\cup\{0\}$ functions
$\gamma_{\mathrm{ext}}$, $\gamma_{\mathrm{int}}$, such that for any $x\in X$, $\hat x\in \hat X$,
$\hat\nu\in\hat{\mathcal{U}}$, $\hat \omega \in\hat{\mathcal{W}}$ there
exists $\nu\in{\mathcal{U}}$ so that for all $\omega \in\mathcal{W}$ and
$t\in\R_{\ge0}$ we have
\begin{IEEEeqnarray}{rCl}\label{inequality}
\begin{IEEEeqnarraybox}[][c]{l}
|\hat \zeta_{\hat x,\hat \nu,\hat\omega}(t)-\zeta_{x,\nu,\omega}(t)|
\le
\beta(V(\hat x, x),t)\\\qquad\qquad\qquad\qquad\qquad
+\:\gamma_{\mathrm{ext}}(||\hat \nu||_\infty)
+\gamma_{\mathrm{int}}(||\omega-\hat \omega||_\infty).
\end{IEEEeqnarraybox}
\end{IEEEeqnarray}
\end{theorem}
The proof, which is given in the appendix, follows the usual arguments that
are known from similar results in the context of
input-to-state Lyapunov functions, e.g. see~\cite{Son89}.

We need the following technical corollary later in the proof of Theorem~\ref{t:lin:nec}.
\begin{corollary}\label{c:2}
Given the assumptions of Theorem~\ref{theorem1}, there
exist a $\mathcal{KL}$ function $\beta$ and $\mathcal{K}\cup\{0\}$ functions
$\gamma_{\mathrm{ext}}$, $\gamma_{\mathrm{int}}$ such that 
for any
$\hat\nu\in\hat{\mathcal{U}}$, $\hat \omega \in\hat{\mathcal{W}}$,
$x\in X$, and $\hat x\in \hat X$ there exists
$\nu\in\mathcal{U}$ so that for every $\omega \in\mathcal{W}$ and
$t\in\R_{\ge0}$ we have
\begin{IEEEeqnarray}{rCl}\label{e:sf}
\begin{IEEEeqnarraybox}[][c]{l}
V(\hat \xi_{\hat x,\hat \nu,\hat \omega}(t), \xi_{x,\nu,\omega}(t))
\le
\beta(V(\hat x, x),t)\\\qquad\qquad\qquad\qquad\qquad
+\:\gamma_{\mathrm{ext}}(||\hat \nu||_\infty)
+\gamma_{\mathrm{int}}(||\omega-\hat \omega||_\infty),
\end{IEEEeqnarraybox}
\end{IEEEeqnarray}
where the $\mathcal{KL}$ function $\beta$
satisfies $\beta(r,0)=r$ for all $r\in\R_{\ge0}$.
\end{corollary}

The proof is provided in the appendix.

\begin{remark}\label{r:interface}
If we are given an \emph{interface function} $k$ that maps every $x$, $\hat x$, $\hat
u$ and $\hat w$ to an input $u=k(x,\hat x,\hat u,\hat w)$ so
that~\eqref{inequality1} is satisfied, then, the input $\nu\in
\mathcal{U}$ that realizes~\eqref{inequality} is readily given
by $\nu(t)=k(\xi(t),\hat\xi(t),\hat\nu(t),\hat\omega(t))$,
see~\cite[Thm.~1]{RZ15}.
\end{remark}
Given an interface function, 
we might exploit the usefulness of simulation functions as follows.
For various reasons (e.g. lower dimension) it might be easier to
synthesize a controller
for the system $\hat \Sigma$ enforcing some complex specifications,
  e.g. given as formulae in linear temporal logic \cite{BK08}, rather than for the original system
$\Sigma$. Then we can use the interface function $k$ to
\emph{transfer} or \emph{refine} the controller that we computed for
$\hat \Sigma$ to a controller for the system $\Sigma$ (cf. example in Section \ref{example}). In this context, we
refer to $\hat\Sigma$ as an \emph{approximate abstraction} and to $\Sigma$ as the
\emph{concrete} system. 
A quantification of the error
that is introduced in the design process by taking the detour through
the abstraction is given by~\eqref{inequality}. A uniform error bound
can be obtained by bounding the difference of the initial states (measured in
terms of $V(\hat x,x)$) together with bounds on the infinity norms of $\hat\nu$ and
$\omega-\hat \omega$.

\begin{remark}
In case that a control system does not have internal inputs, the 
definition~\eqref{e:sys} reduces to $\left( X, U,\mathcal{U},f,Y,h\right)$ and the vector
field becomes $f:X\times U\to\R^n$. 
Correspondingly, the definition of simulation functions
  simplifies, i.e., 
in~\eqref{inequality1} we do not quantify the inequality over
$w,\hat w$ and the term \mbox{$\mu(|w_1-\hat w_1|,\ldots,|w_p-\hat
w_p|)$} is omitted.
Similarly, the results in Theorem~\ref{theorem1} and Corollary~\ref{c:2}
are modified, i.e., 
inequalities~\eqref{inequality} and~\eqref{e:sf} are not quantified
over $\omega,\hat \omega\in\mathcal{W}$ and the term
$\gamma_{\mathrm{int}}(||\omega-\hat \omega||_\infty)$ is omitted.
\end{remark}

\section{Compositionality Result}
\label{s:inter}

In this section, we analyze interconnected control systems and show
how to construct an approximate abstraction of an
interconnected system and the corresponding simulation function from the abstractions of the subsystems and their corresponding simulation functions, respectively. The
definition of the interconnected control system is based on the notion of
interconnected systems introduced in~\cite{TI08}.

\subsection{Interconnected Control Systems}

We consider $N\in\N_{\ge1}$ control systems
\begin{IEEEeqnarray*}{c'c}
\Sigma_i=\left(X_i,U_i,W_i,\mathcal{U}_i,\mathcal{W}_i,f_i,Y_i,h_i\right),
&i\in\intcc{1;N}
\end{IEEEeqnarray*}
with partitioned internal inputs and outputs
\begin{IEEEeqnarray}{c'c}\label{e:iop}
\begin{IEEEeqnarraybox}[][c]{c,c}
w_i=(w_{i1};\ldots;w_{i(i-1)};w_{i(i+1)};\ldots;w_{iN}), \\
y_i=(y_{i1};\ldots;y_{iN}),
\end{IEEEeqnarraybox}
\end{IEEEeqnarray}
with $w_{ij}\in W_{ij}\subseteq \R^{p_{ij}}$, $y_{ij}\in Y_{ij}\subseteq
\R^{q_{ij}}$,
and output function
\begin{IEEEeqnarray}{c'c}\label{e:of}
h_i(x_i)=(h_{i1}(x_i);\ldots;h_{iN}(x_i)),
\end{IEEEeqnarray}
as depicted schematically in Figure~\ref{system}.
\begin{figure}[ht]
\begin{center}
\begin{tikzpicture}[>=latex',scale=1]
  \draw[thick]  (0,0) rectangle node {$\Sigma_i$} (1.8,1);

  \draw[->]   (-1,.9) node[left] {\footnotesize $u_i$} -- (0,.9);
  \draw[->]   (-1,.7) node[left] {\footnotesize $w_{i1}$} -- (0,.7);
  \draw[->]   (-1,.1) node[left] {\footnotesize $w_{iN}$} -- (0,.1);
  \path   (-1,.5) -- node {$\vdots$} (0,.5);

  \draw[->]   (1.8,.9) -- node[very near end,right] {\footnotesize $y_{i1}$} (2.8,.9);
  \draw[->]   (1.8,.7) -- node[very near end,right] {\footnotesize $y_{i2}$} (2.8,.7);
  \draw[->]   (1.8,.1) -- node[very near end,right] {\footnotesize $y_{iN}$} (2.8,.1);
  \path       (1.8,.5) -- node {$\vdots$}        (2.8,.5);

\end{tikzpicture}
\end{center}
\vspace{-0.3cm}
\caption{Input/output configuration of subsystem~$\Sigma_i$.}
\label{system}
\vspace{-0.3cm}
\end{figure}
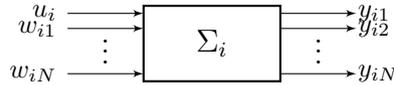

We interpret the outputs $y_{ii}$ as \emph{external} outputs, whereas the outputs $y_{ij}$ with $i\neq j$ are \emph{internal}
outputs which are used to define the
interconnected systems. In particular, we assume that the
dimension of $w_{ij}$ is equal to the dimension of $y_{ji}$, i.e.,
the following \emph{interconnection constraints} hold:
\begin{IEEEeqnarray}{c'c'c}\label{e:ic}
 \forall i,j\in\intcc{1;N},~i\neq j:& q_{ij}=p_{ji},& Y_{ij}\subseteq W_{ji}.
\end{IEEEeqnarray}
If there is no connection from subsystem $\Sigma_{i}$ to
$\Sigma_j$, we simply set $h_{ij}\equiv 0$.

\begin{definition}\label{d:ics}
Consider $N\in\N_{\ge1}$ control systems $
\Sigma_i=\left(X_i,U_i,W_i,\mathcal{U}_i,\mathcal{W}_i,f_i,Y_i,h_i\right)$,
$i\in\intcc{1;N}$, with the input-output structure given
by~\eqref{e:iop}-\eqref{e:ic}. The \emph{interconnected control
system} $\Sigma=\left(X,U,\mathcal{U},f,Y,h\right)$,
denoted by
$\mathcal{I}(\Sigma_1,\ldots,\Sigma_N)$, is given by 
$X:=X_1\times\cdots\times X_N$, $U:=U_1\times\cdots\times U_N$,
$Y:=Y_{11}\times\cdots\times Y_{NN}$ and functions
\begin{IEEEeqnarray*}{c}
f(x,u):=(f_1(x_1,u_1,w_1);\ldots;f_N(x_N,u_N,w_N)),\\
h(x):=(h_{11}(x_1);\ldots;h_{NN}(x_N)),
\end{IEEEeqnarray*}
where $u=(u_{1};\ldots;u_{N})$ and $x=(x_{1};\ldots;x_{N})$ and
with the interconnection variables constrained by
$w_{ij}=y_{ji}$ for all $i,j\in\intcc{1;N},i\neq j$.
\end{definition}

An example of an interconnection of two control subsystems $\Sigma_1$
and $\Sigma_2$ is illustrated in Figure \ref{system1}.
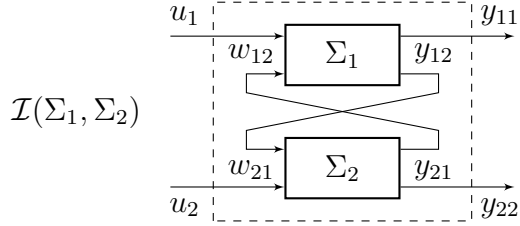
\begin{figure}[ht]
\begin{tikzpicture}[>=latex']
\tikzstyle{block} = [draw, 
                     thick,
                     rectangle, 
                     minimum height=.8cm, 
                     minimum width=1.5cm]

  \node at (-3.5,-0.75) {$\mathcal{I}(\Sigma_1,\Sigma_2)$};

  \draw[dashed] (-1.7,-2.2) rectangle (1.7,.7);

  \node[block] (S1) at (0,0) {$\Sigma_1$};
  \node[block] (S2) at (0,-1.5) {$\Sigma_2$};

  \draw[->] ($(S1.east)+(0,0.25)$) -- node[very near end,above] {$y_{11}$} ($(S1.east)+(1.5,.25)$);
  \draw[<-] ($(S1.west)+(0,0.25)$) -- node[very near end,above] {$u_{1}$} ($(S1.west)+(-1.5,.25)$);

  \draw[->] ($(S2.east)+(0,-.25)$) -- node[very near end,below] {$y_{22}$} ($(S2.east)+(1.5,-.25)$);
  \draw[<-] ($(S2.west)+(0,-.25)$) -- node[very near end,below] {$u_{2}$} ($(S2.west)+(-1.5,-.25)$);

  \draw[->] 
    ($(S1.east)+(0,-.25)$) -- node[very near end,above] {$y_{12}$} 
    ($(S1.east)+(.5,-.25)$) --
    ($(S1.east)+(.5,-.5)$) --
    ($(S2.west)+(-.5,.5)$) --
    ($(S2.west)+(-.5,.25)$) -- node[very near start,below] {$w_{21}$}
    ($(S2.west)+(0,.25)$) ;

  \draw[->] 
    ($(S2.east)+(0,.25)$) -- node[very near end,below] {$y_{21}$} 
    ($(S2.east)+(.5,.25)$) --
    ($(S2.east)+(.5,.5)$) --
    ($(S1.west)+(-.5,-.5)$) --
    ($(S1.west)+(-.5,-.25)$) -- node[very near start,above] {$w_{12}$}
    ($(S1.west)+(0,-.25)$) ;

\end{tikzpicture}
\vspace{-0.2cm}
\caption{Interconnection of two control subsystems $\Sigma_1$ and $\Sigma_2$.}
\label{system1}
\vspace{-0.3cm}
\end{figure}

\subsection{Compositional Construction of Approximate Abstractions and Simulation Functions}

In this subsection, we assume that we are given $N$ subsystems
$\Sigma_i=\left(X_i,U_i,W_i,\mathcal{U}_i,\mathcal{W}_i,f_i,Y_i,h_i\right),$
together with their abstractions 
$\hat\Sigma_i=(\hat X_i,\hat U_i,\hat W_i,\hat{\mathcal{U}}_i,\hat{\mathcal{W}_i},\hat f_i,\hat Y_i,\hat h_i)$
 and the simulation functions $V_i$ from
$\hat\Sigma_i$ to $\Sigma_i$,
with the associated comparison functions denoted
by $\alpha_i$, $\lambda_i$, $\rho_{i}$ and  $\mu_i$. We assume that
the arguments of $\mu_i$ are partitioned according to the interconnection
scheme, i.e., $\mu_i\in$ MAF$_{N-1}$ and the internal inputs appear in~\eqref{inequality1}
for $\Delta w_{ij}:=|w_{ij}-\hat w_{ij}|$
according to
\begin{IEEEeqnarray}{c}\label{e:maf:i}
\mu_i(\Delta w_{i1},\ldots,\Delta w_{i(i-1)},\Delta w_{i(i+1)},\ldots,\Delta
w_{iN}).
\IEEEeqnarraynumspace
\end{IEEEeqnarray}
We follow~\cite{DIW11} and use an operator $\Gamma:\R_{\ge0}^N\to
\R^N_{\ge0}$ to formulate a small gain condition.
Each component $\Gamma_i$ with $r_j:=\alpha_j^{-1}(s_j)$ is given by
\begin{IEEEeqnarray}{c}\label{e:Goperator}
\Gamma_i(s):=
\begin{cases}
\mu_1(r_2,\ldots,r_N)&i=1\\
\mu_i(r_1,\ldots,r_{i-1},r_{i+1},\ldots,r_N),&i\in\intoo{1;N}\\
\mu_N(r_1,\ldots,r_{N-1})&i=N.\\
\end{cases}
\IEEEeqnarraynumspace
\end{IEEEeqnarray}
For $\varepsilon_i\in\mathcal{K}_\infty$, $i\in\intcc{1;N}$  we introduce
$D:\R^N_{\ge0}\to \R^N_{\ge0}$ with 
$D(s)=(s_1+\varepsilon_1(s_1);\ldots ;s_N+\varepsilon_N(s_N))$
as well as $\Lambda^{-1}:\R^N_{\ge0}\to \R^N_{\ge0}$ by
$\Lambda^{-1}(s)=(\lambda_1^{-1}(s_1);\ldots ;\lambda_N^{-1}(s_N))$. The nonlinear
\emph{small gain condition} is given by $D\circ \Gamma\circ
\Lambda^{-1}\not\ge \id$, i.e., for any $s\in\R_{>0}^N$, at least one component of $D\circ \Gamma\circ
\Lambda^{-1}(s)$ is strictly less than the corresponding component of~$s$. One
of the main results in~\cite{DashkovskiyRuefferWirth10} shows that if $D\circ
\Gamma\circ \Lambda^{-1}$ is irreducible and satisfies the small gain condition,
then there exist $\mathcal{K}_\infty$ functions $\sigma_i$, $i\in\intcc{1;N}$
so that 
$\sigma(r)=(\sigma_1(r);\ldots;\sigma_N(r))$ satisfies\footnote{We interpret the
inequality~\eqref{e:path}
component-wise, i.e.,
for $x\in\R^N$ we have $x<0$ iff every entry $x_i<0$,
  $i\in\intcc{1;N}$.}
\begin{IEEEeqnarray}{c't'c}\label{e:path}
  D\circ \Gamma \circ \Lambda^{-1}(\sigma(r))< \sigma(r) & for all &
  r\in\R_{>0}.
\end{IEEEeqnarray}
Subsequently, we term $N$ $\mathcal{K}_\infty$ functions $\sigma_i$ that
satisfy~\eqref{e:path} for some $D$ as
\emph{$\Omega$-path}~\cite[Def.~5.1]{DashkovskiyRuefferWirth10}.

Suppose that $N=2$ and $\alpha_i=\id$ for $i\in\intcc{1;2}$. The
small gain condition requires  that there exist $\varepsilon_i\in \mathcal{K}_\infty$ so
that either 
$(\id +\varepsilon_1)\circ \mu_{1}\circ \lambda_2^{-1}(s_2)<s_1 $ 
or 
$(\id +\varepsilon_2)\circ \mu_{2}\circ \lambda_1^{-1}(s_1)<s_2 $
holds for all $s\in\R^2_{>0}$. This
follows, e.g. by the small gain condition used in \cite{JiangMareelsWang96}
\begin{IEEEeqnarray}{c,c}\label{e:sg:JiangMareelsWang96}
  \exists_{\varepsilon\in\R_{>0}}\forall_{r\in\R_{>0}}&(1 +\varepsilon) \mu_{2}\circ
\lambda_1^{-1} \circ (1 +\varepsilon)\mu_{1}\circ \lambda_2^{-1}(r)<r.
\IEEEeqnarraynumspace
\end{IEEEeqnarray}
The main technical result in \cite{JiangMareelsWang96}, which enables the small
gain theorem, shows that \eqref{e:sg:JiangMareelsWang96} implies the existence
of $\sigma_2\in \mathcal{K}_\infty$ so that $(1+\varepsilon)\mu_{2}\circ
\lambda_1^{-1} (r) < \sigma_2(r) < \lambda_2\circ \mu_{1}^{-1}(\frac{r}{1
+\varepsilon})$ holds  for all $r>0$. It is easy to check that
$\sigma(s)=(s_1,\sigma_2(s_2))$
satisfies~\eqref{e:path}. In the context of simulation functions,
condition~\eqref{e:sg:JiangMareelsWang96} ensures that the output mismatch propagated through the
interconnected systems is not amplified.
For general interconnected systems, the small gain condition can also be interpreted
as the requirement that the ``loop-gains'' associated with the cycles of the
interconnection graph are strictly less than one,
see~\cite[Sec.~8.4]{DashkovskiyRuefferWirth10}.

If the functions $\alpha_i$, $\mu_i$ and $\lambda_i$ are linear,
the existence of an $\Omega$-path 
follows from $\Gamma \Lambda^{-1}$ having spectral radius strictly less than
one~\cite[Thm.~5.1]{DashkovskiyRuefferWirth10}. In this case, the right eigenvector  $\eta\in\R^N_{>0}$
associated with the spectral radius has positive entries, and it follows that 
$D\Gamma \Lambda^{-1} \eta< \eta$ for some appropriately picked  $\varepsilon_i>0$.
Hence, 
$\sigma(r)=(\eta_1r;\ldots;\eta_Nr)$, is an $\Omega$-path.

In the following theorem, similar to~\cite[Thm~4.5]{DIW11}, we use the technical
assumption on the derivative $(\sigma_i^{-1}\circ \lambda_i)'$ of the functions
$\sigma_i^{-1}\circ \lambda_i$ which reads
\begin{IEEEeqnarray}{c}\label{e:sigmalambda}
\begin{IEEEeqnarraybox}[][c]{l}
\forall_{i\in\intcc{1;N}}
\forall_{\kappa\in\mathcal{K}_\infty}\exists_{\underline
\kappa\in\mathcal{K}_\infty}\forall_{r\in\R_{>0}}:
\underline\kappa(r)\le \kappa(r)(\sigma_i^{-1}\circ \lambda_i)'(r)\\
\forall_{i\in\intcc{1;N}}
\forall_{\kappa\in\mathcal{K}}\exists_{\overline\kappa\in\mathcal{K}}\forall_{r\in\R_{>0}}:
\kappa(r)(\sigma_i^{-1}\circ \lambda_i)'(r)\le
\overline\kappa(r).
\end{IEEEeqnarraybox}
\IEEEeqnarraynumspace
\end{IEEEeqnarray}
\begin{theorem}\label{t:ic}
Consider the interconnected control system
$\Sigma=\mathcal{I}(\Sigma_1,\ldots,\Sigma_N)$ induced by
$N\in\N_{\ge1}$
control subsystems~$\Sigma_i$.
Suppose that for each subsystem $\Sigma_i$, we are given $\hat
\Sigma_i$ together with a simulation
function $V_i$ from $\hat \Sigma_i$ to $\Sigma_i$ with comparison functions
$\alpha_i$, $\lambda_i$, $\rho_{i}$ and $\mu_i$.
Suppose that there exists an $\Omega$-path
$\sigma$ and for every $i\in\intcc{1;N}$ $\sigma_i^{-1}\circ
\lambda_i$ is differentiable on $\R_{>0}$ and~\eqref{e:sigmalambda}
holds.
Then 
\begin{IEEEeqnarray}{c}\label{e:icSF}
V(\hat x, x)=\max_{i\in\intcc{1;N}}\{\sigma_i^{-1}\circ \lambda_i\circ V_i(\hat x_i, x_i)\}
\end{IEEEeqnarray}
is a simulation function from $\hat \Sigma=\mathcal{I}(\hat
\Sigma_1,\ldots,\hat\Sigma_N)$ to $\Sigma$. 
\end{theorem}
\begin{proof}
We follow the arguments in~\cite[Thm.~4.5]{DIW11}.
Let us first point out, that $\sigma_i^{-1}\circ \lambda_i$ and $V_i$ being
differentiable and locally Lipschitz, respectively, implies 
that $V$ is locally Lipschitz. 
Let us show inequality \eqref{e:bsf:1} for
\mbox{$x=(x_1;\ldots;x_N)\in X$} and 
$\hat x=(\hat x_1;\ldots;\hat x_N)\in \hat X$. We derive
\begin{IEEEeqnarray*}{rCl}
\vert \hat h(\hat x)-h(x) \vert
&\le&
\sqrt{N} \max_{i\in\intcc{1;N}} \vert \hat h_{ii}(\hat x_i)-h_{ii}(x_i) \vert\\
&\le&
\sqrt{N}\max_{i\in\intcc{1;N}}  \alpha_i^{-1}\circ V_i(\hat x_i, x_i)\\
&\le&
\bar\alpha(\max_{i\in\intcc{1;N}} \sigma^{-1}_i\circ \lambda_i\circ V_i(\hat
x_i, x_i))=
\bar\alpha( V(\hat x, x))
\end{IEEEeqnarray*}
where 
$\bar\alpha(r)=\sqrt{N}\max_i\alpha^{-1}_i\circ\lambda_i^{-1}\circ \sigma_i(r)$
which is a $\mathcal{K}_\infty$ function and \eqref{e:bsf:1} holds with
$\alpha=\bar\alpha^{-1}$. 

We continue with showing \eqref{inequality1}. Let $z,v\in \hat X\times X$. Using a straightforward extension of~\cite[Thm.~1]{MilgromSegal02}, we obtain
\begin{IEEEeqnarray}{c}\label{e:ic:p1}
\mathsf{D}^+V(z,v)\le\max \{\mathsf{D}^+(\sigma_i^{-1}\circ\lambda_i\circ
V_i)(z,v)\mid i\in I(z)\}
\IEEEeqnarraynumspace
\end{IEEEeqnarray}
where $I(z)=\{i\in\intcc{1;N}\mid V(z)=\sigma_i^{-1}\circ\lambda_i\circ
V_i(z)\}$. Moreover, by Lemma~\ref{l:chainrule} in the appendix, we have 
\begin{IEEEeqnarray}{c}\label{e:ic:p2}
\mathsf{D}^+(\sigma_i^{-1}\circ\lambda_i\circ V_i)(z,v)\le
(\sigma_i^{-1}\circ\lambda_i)'(V_i(z))\mathsf{D}^+ V_i(z,v).
\IEEEeqnarraynumspace
\end{IEEEeqnarray}
We fix 
$x=(x_1;\ldots;x_N)$,
$\hat x=(\hat x_1;\ldots;\hat x_N)$,
in $\hat X\times X\smallsetminus V_0$, $\hat u=(\hat u_{1};\ldots;\hat u_{N})\in\hat U$ and
$u=(u_{1};\ldots;u_{N})\in U$,
where we pick $u_i$ to satisfy~\eqref{inequality1} 
with the internal inputs given by $w_{ij}=h_{ji}(x_j)$ and $\hat
w_{ij}=\hat h_{ji}(\hat x_j)$. We define $\Delta w_{ij}:=|w_{ij}-\hat w_{ij}|$, $\Delta y_{ji}:=|y_{ji}-\hat y_{ji}|$
and $V_{\mathrm{vec}}=(V_1;\ldots; V_N)$, then we get
\begin{IEEEeqnarray*}{l}
\mu_i(\Delta w_{i1},\ldots,\Delta w_{i(i-1)},\Delta w_{i(i+1)},\ldots,\Delta w_{iN})\\
=\mu_i(\Delta y_{1i},\ldots,\Delta y_{(i-1)i},\Delta y_{(i+1)i},\ldots,\Delta
y_{Ni})\\
\le\mu_i(\alpha_1^{-1}(V_1),\ldots,\alpha_{i-1}^{-1}(V_{i-1}),\alpha_{i+1}^{-1}(V_{i+1}),\ldots,\alpha_N^{-1}(V_N))\\
\le\Gamma_i(V_{\mathrm{vec}}).
\end{IEEEeqnarray*}
Moreover, we see that $\Gamma_i(V_{\mathrm{vec}})$ equals  
\begin{IEEEeqnarray*}{rCl}
\Gamma_i\circ \Lambda^{-1}(
\sigma_1\circ \sigma_1^{-1}\circ \lambda_1(V_1);\ldots;\sigma_N\circ\sigma_N^{-1}\circ \lambda_N(V_N)).
\end{IEEEeqnarray*}
Using~\eqref{e:path}, i.e., $\Gamma\circ \Lambda^{-1}(\sigma(r))<D^{-1}\circ
\sigma(r)$, and~\eqref{e:icSF} we obtain a bound of~\eqref{e:maf:i} by
$(\id+\varepsilon_i)^{-1}\circ\sigma_i(V)$.
Let us slightly abuse notation and use
$V_i$ and $\mathsf{D}^+V_i$,  for $V_i(\hat
x_i,x_i)$ and  $\mathsf{D}^+V_i((\hat x_i,x_i),(\hat f_i(\hat x_i,\hat u_i,\hat w_i),f_i(x_i,u_i,w_i)))$. 
Similarly, we simplify the notation for $V$ and $\mathsf{D}^+V$.
Let $i\in I((\hat x, x))$, then we compute 
\begin{IEEEeqnarray*}{rCl}
\mathsf{D}^+V_i
&\le&
-\lambda_i(V_i)+(\id+\varepsilon_i)^{-1}\circ\sigma_i(V)+\rho_{i}(|\hat u_i|)\\
&\le&
-\sigma_i(V)+(\id+\varepsilon_i)^{-1}\circ\sigma_i(V)+\rho_{i}(|\hat u_i|)\\
&\le&
-\varepsilon_i\circ (\id +\varepsilon_i)^{-1}\circ\sigma_i(V)+\rho_{i}(|\hat u_i|).
\end{IEEEeqnarray*}
Using~\eqref{e:sigmalambda}, it follows that there exist $\underline\kappa_i\in\mathcal{K}_\infty$
 and $\overline\kappa_i\in\mathcal{K}\cup\{0\}$ so that 
$\underline\kappa_i\circ\tfrac{\varepsilon_i}{\id +\varepsilon_i}\circ\sigma_i(r)\le
(\sigma_i^{-1}\circ \lambda_i)'(r)\tfrac{\varepsilon_i}{\id
+\varepsilon_i}\circ\sigma_i(r)$
and
$(\sigma_i^{-1}\circ \lambda_i)'(r)\rho_{i}(r)\le
 \overline\kappa_i\circ \rho_{i}(r)$ 
holds for all $r\in\R_{>0}$. We define $\lambda\in\mathcal{K}_\infty$ and
$\rho\in\mathcal{K}\cup\{0\}$ by
\begin{IEEEeqnarray*}{c,c}
\lambda(r)=\min_{i\in\intcc{1;N}}\{ \underline\kappa_i\circ\tfrac{\varepsilon_i}{\id
+\varepsilon_i}\circ\sigma_i(r)\},&
\rho(r)=\max_{i\in\intcc{1;N}}\{
  \overline\kappa_i\circ\rho_{i}(r)\}.
\end{IEEEeqnarray*}
Using~\eqref{e:ic:p1} and~\eqref{e:ic:p2} we get
$\mathsf{D}^+V\le-\lambda(V)+\rho(|\hat u|)$ which completes the proof.
\end{proof}
\begin{remark}\label{e:comp:bound:lin}
In the linear case, with $\sigma(r)=(\eta_1r;\ldots;\eta_Nr)$ we get $V(\hat
x,x)=\max_i\tfrac{\lambda_i}{\eta_i}V_i(\hat x_i,x_i)$ with
$\lambda=\min_i\{\lambda_i\tfrac{\varepsilon_i}{1+\varepsilon_i}\}$ and
$\rho_{}=\max_i\{\rho_{i}\tfrac{\lambda_i}{\eta_i}\}$,
where we abuse notation and identify linear functions $\alpha(r)=a r$ with their
coefficients, i.e., $\alpha=a$.
\end{remark}

\section{Approximate Abstractions and\\ Simulation Functions for Linear Systems}
\label{s:lin}

In this section, we focus on \emph{linear} control systems $\Sigma$ and
\emph{square-root-of-quadratic} simulation functions $V$. In the first
part, we follow the geometric approach to linear control systems, and 
characterize simulation functions for
\emph{linear} control systems $\Sigma$ in terms of \emph{controlled invariant externally
stabilizable subspaces}~\cite{BM92}. The results
are closely connected to the characterization of simulation
relations developed in~\cite{vdS04}.  
In the second part, we use the characterization of simulation
functions to actually construct abstractions
of linear control subsystems whose existence was assumed in the first part of the paper.

\subsection{Characterization of Simulation Functions}

A \emph{linear control system} is defined as a control system with 
the vector field and output function given by the following linear maps
\begin{IEEEeqnarray}{c}\label{e:lin:sys}
\begin{IEEEeqnarraybox}[][c]{rCl}
\dot \xi(t)&=&A\xi(t)+B\nu(t)+D\omega(t),\\
\zeta(t)&=&C\xi(t),
\end{IEEEeqnarraybox}
\end{IEEEeqnarray}
with the state space, external input space, internal input space and
output space given by $\R^n$, $\R^m$, $\R^p$ and $\R^q$, respectively.
The dimensions of the matrices follow by
\begin{IEEEeqnarray}{c,c,c,t,c}\label{e:lin:mat}
A\in\R^{n\times n},&
B\in\R^{n\times m},&
D\in\R^{n\times p},& and &
C\in\R^{q\times n}.
\end{IEEEeqnarray}
Henceforth, we simply use the tuple
$\Sigma=(A,B,C,D)$
to refer to a control system with vector field and output function of
the form of~\eqref{e:lin:sys} with the dimension of the corresponding matrices
specified by~\eqref{e:lin:mat}. 
As the co-domain of the internal and external inputs are
implicitly determined by the dimension of $B$ and~$D$, we do not
include the sets $\mathcal{U}$ and $\mathcal{W}$ in the system
tuple.

In the following we characterize simulation functions from
$\Sigma_1=(A_1,B_1,C_1,D_1)$ to $\Sigma_2=(A_2,B_2,C_2,D_2)$ in terms
of the auxiliary matrices given by
\begin{IEEEeqnarray}{c}\label{e:auxmat}
\begin{IEEEeqnarraybox}[][b]{lc;c;c;c}
A_{12}=\begin{bmatrix} A_1& 0\\ 0 & A_2 \end{bmatrix},&
B_{12}=\begin{bmatrix} 0\\ B_2\end{bmatrix},& 
B_{21}=\begin{bmatrix} B_1\\ 0\end{bmatrix},&
D_{12}=\begin{bmatrix} D_1\\ D_2\end{bmatrix},\\
 C_{12}=\begin{bmatrix} -C_1& C_2 \end{bmatrix}.
\end{IEEEeqnarraybox}
\end{IEEEeqnarray}

\begin{theorem}[Necessity]\label{t:lin:nec}
Consider two linear control systems $\Sigma_i=(A_i,B_i,C_i,D_i)$, $i\in\{1,2\}$ with
the same internal input space dimension and the same output space dimension.
Let the matrices $A_{12}, B_{12},B_{21},C_{12},D_{12}$ be given
by~\eqref{e:auxmat}.
Suppose there exists a simulation function $V$ from
$\Sigma_1$ to $\Sigma_2$, then 
there exists a relation
$R\subseteq \R^{n_1}\times \R^{n_2}$
which is a subspace that satisfies
  \begin{IEEEeqnarray}{s}
    \IEEEyesnumber
    \label{e:lin:nec}
    \IEEEyessubnumber
    \label{e:lin:nec:estab}
    $R$ is $(A_{12},B_{12})$-externally stabilizable\\
    \IEEEyessubnumber
    \label{e:lin:nec:cinv}
    $A_{12}R\subseteq R+\im B_{12}$\\
    \IEEEyessubnumber
    \label{e:lin:nec:D}
    $\im D_{12}\subseteq R +\im B_{12}$\\
    \IEEEyessubnumber
    \label{e:lin:nec:ker}
    $R\subseteq \ker C_{12}$.
  \end{IEEEeqnarray}
If the function 
$\rho$ associated with $V$ equals
to zero, then
  \begin{IEEEeqnarray}{c}
    \IEEEyessubnumber
    \label{e:lin:nec:exactSR}
    \im B_{21}\subseteq R+\im B_{12}.
  \end{IEEEeqnarray}
\end{theorem}
\begin{proof}[Proof of Theorem~\ref{t:lin:nec}]
Let $R\subseteq \R^{n_1}\times \R^{n_2}$ be the smallest subspace in
$\R^{n_1\times n_2}$ that contains the set $S=\{(x_1;x_2)\mid
V(x_1,x_2)=0\}$. By definition, any element of $R$ 
follows by applying scalar multiplication and addition to elements in
$S$, and therefore, we obtain $R\subseteq \ker C_{12}$.
Now let $\nu_1\equiv 0$ and
$\omega_1=\omega_2$. Choose $x_1, x_2$ such that 
$V(x_1, x_2)=0$, 
then it follows from Corollary~\ref{c:2} that there exists
$\nu_2$ such that 
$V(\xi_{1,x_1,\nu_1,\omega_1}(t),\xi_{2,x_2,\nu_2,\omega_2}(t))=0$ holds for all $t\in\R_{\ge0}$. By the linearity of solutions of
linear systems, we have $(x_1,x_2)\in R$ implies that there exists
$\nu_2$ such that 
\mbox{$(\xi_{1,x_1,\nu_1,\omega_1}(t),\xi_{2,x_2,\nu_2,\omega_2}(t))\in R$}
holds for all $t\in\R_{\ge0}$,
which shows that $R$ is $(A_{12},B_{12})$-controlled invariant,
see~\cite[Thm.~4.1.1]{BM92} and~\eqref{e:lin:nec:cinv} follows.
As the choice of $x_1$, $x_2$ and $\omega_1$ is
arbitrary, we invoke the fundamental lemma of the geometrical
approach~\cite[Lem.~3.2.1]{BM92} and obtain that for every $x_1$, $x_2$, $w_1$ there is
$u_2$ so that
$(A_1 x_1+D_1w_1,A_2x_2+D_2w_1+B_2u_2)\in R$. By setting $x_1=x_2=0$, we
obtain~\eqref{e:lin:nec:D}.
We
continue to show~\eqref{e:lin:nec:estab}. For $\nu_1\equiv 0$,
$\omega_1=\omega_2$, Corollary~\ref{c:2}
implies that for every $(x_1,x_2)\in \R^{n_1}\times \R^{n_2}$,
there exists
$\nu_2\in\mathcal{U}_2$ such that $\lim_{t\to \infty} V(\xi_{1,x_1,\nu_1,\omega_1}(t),\xi_{2,x_2,\nu_2,\omega_2}(t))=0$, which implies that $(\xi_{1,x_1,\nu_1,\omega_1}(t),\xi_{2,x_2,\nu_2,\omega_2}(t))$ converges to $R$. Then
using~\cite[Def.~4.1.6 and Prp.~4.1.14]{BM92} we conclude that
$R$ is $(A_{12},B_{12})$-externally stabilizable.

We continue with~\eqref{e:lin:nec:exactSR}.
Let $\omega_1=\omega_2\equiv 0$. Since \mbox{$\gamma_{\mathrm{ext}}\equiv
0$}, we use the same arguments as above, and obtain that for every 
\mbox{$(x_1,x_2)\in R$} and \mbox{$\nu_1\in{\mathcal{U}_1}$} there is
\mbox{$\nu_2\in\mathcal{U}_2$} so that 
\mbox{$(\xi_{1,x_1,\nu_1,\omega_1}(t),\xi_{2,x_2,\nu_2,\omega_2}(t))\in R$}
holds for all $t\in\R_{\ge0}$.
By the fundamental lemma of the geometric approach~\cite[Lem.~3.2.1]{BM92}, we have that 
$(\dot{ \xi}_{1,x_1,\nu_1,\omega_1}(t),\dot
\xi_{2,x_2,\nu_2,\omega_2}(t))\in R$ for almost all $t\in\R_{\ge0}$.
This implies, for every $(x_1,x_2)\in R$ and $u_1\in
\R^{m_1}$ there exists $u_2\in \R^{m_2}$ so that  $(A_1x_1+ B_1
u_1,A_2x_2+B_2u_2)\in R$, which concludes the proof.
\end{proof}

\begin{theorem}[Sufficiency]\label{t:lin:suf}
Consider two linear control systems $\Sigma_i=(A_i,B_i,C_i,D_i)$,
$i\in\{1,2\}$ with
the same internal input space dimension and the same output space dimension.
Let the matrices $A_{12}, B_{12},B_{21},C_{12},D_{12}$ be given by~\eqref{e:auxmat}.
Suppose there exists a linear subspace
 $R\subseteq \R^{n_1}\times \R^{n_2}$
that satisfies~\eqref{e:lin:nec:estab}-\eqref{e:lin:nec:ker}, then 
there exists a symmetric positive semi-definite matrix 
$M\in\R^{(n_1+n_2)\times(n_1+n_2)}$
so that 
\begin{IEEEeqnarray*}{c}
  V(x_1, x_2)=\big((x_1; x_2)^\top M(x_1; x_2)\big)^{\tfrac{1}{2}}
\end{IEEEeqnarray*}
is a simulation function
from $\Sigma_1$ to $\Sigma_2$.
If additionally~\eqref{e:lin:nec:exactSR} holds,
then the function 
$\rho$ associated with $V$ equals
to zero, i.e., $\rho\equiv0$.
\end{theorem}
\begin{proof}[Proof of Theorem~\ref{t:lin:suf}]
We pick $K_{12}$ so that $(A_{12}+B_{12}K_{12})R\subseteq R$
and $(A_{12}+B_{12}K_{12})|_{(\R^{n_1}\times \R^{n_2})/R}$ is
Hurwitz. Let $A=(A_{12}+B_{12}K_{12})$, from~\cite[Proof of Thm~3.2.1 and
Def.~3.2.4]{BM92} it follows that for any invertible matrix
$T=[T_1\;T_2]$ with
$\im T_1=R$ and $T^{-1}=[\bar T_1^\top\; \bar T_2^\top]^\top$ we obtain
\begin{IEEEeqnarray}{c} \label{e:lin:suf:structure}
\begin{IEEEeqnarraybox}[][c]{c}
T^{-1}A T=
\begin{bmatrix}
F_{11} & F_{12}\\
0 & F_{22}
\end{bmatrix}, \text{$F_{22}$ is Hurwitz, }
F_{22} \bar T_2=\bar T_2 A.
\end{IEEEeqnarraybox}
\IEEEeqnarraynumspace
\end{IEEEeqnarray}
For the remainder, we use $\bar A=F_{22}$ and $\Pi=\bar T_2$. 
Let $x\in \ker \bar T_2$, then we compute $x=TT^{-1}x=T_1y$ for
$y=\bar T_1x$ and it follows that $\ker \Pi\subseteq R$. Since
$R\subseteq \ker C_{12}$, we obtain $\ker \Pi\subseteq \ker
C_{12}$ and there exists $\bar C$ so that $\bar C\Pi=C_{12}$.  As $\bar A$ is Hurwitz, there exist a constant $\lambda\in\R_{>0}$
and a symmetric positive
definite matrix $\bar M$, so that 
\begin{IEEEeqnarray}{l}\label{e:t:lin:suf:1}
\begin{IEEEeqnarraybox}[][c]{l}
\bar C^\top \bar C\le \bar M\\
\bar A^\top \bar M+\bar M\bar A\le -2\lambda \bar M.
\end{IEEEeqnarraybox}
\end{IEEEeqnarray}
We define $V:\R^{n_1}\times \R^{n_2}\to \R_{\ge0}$ by
\begin{IEEEeqnarray*}{l}
  V(x_1,x_2)=\big((x_1;x_2)^\top \Pi^\top \bar M \Pi (x_1;x_2)\big)^{\tfrac{1}{2}}.
\end{IEEEeqnarray*}
Clearly $M=\Pi^\top \bar M\Pi$ is symmetric positive semi-definite and it remains to
show that $V$ is indeed a simulation
function from $\Sigma_1$ to $\Sigma_2$. First, we
verify that~\eqref{e:bsf:1} holds for $\alpha= \id$ by
\begin{IEEEeqnarray*}{rCl}
|C_1x_1-C_2x_2|^2&=&|C_{12}(x_1;x_2)|^2 =|\bar C\Pi(x_1;x_2)|\\
&\le& (x_1;x_2)^\top \Pi^\top\bar M\Pi(x_1;x_2)= V(x_1,x_2)^2.
\end{IEEEeqnarray*}
We continue to show that~\eqref{inequality1} holds as well. Let $x_1$, $x_2$, $u_1$ and
$w_1$ be given. Then we pick $u_2=K_{12}(x_1;x_2)+K_4u_1+u_3$ where we
pick $u_3$ so
that $D_{12} w_1+B_{12} u_{3}\in R$ holds,
which is possible by~\eqref{e:lin:nec:D}. The purpose of $K_4u_1$ will
become apparent later.
Let $x=(x_1;x_2)$, then for any $w_2$
the left-hand-side of~\eqref{inequality1} evaluates to 
\begin{IEEEeqnarray*}{c}
\frac{x^\top \Pi^\top \bar M \Pi}{V(x_1,x_2)}
\bigg[
A x
+
D_{12} w_1+ B_{12}u_3
+
\begin{bmatrix}B_1\\B_2K_4 \end{bmatrix}
u_1
+
\begin{bmatrix}0\\D_2 \end{bmatrix}
\Delta w
\bigg]
\end{IEEEeqnarray*}
with $\Delta w=(w_1- w_2)$.
We use $\bar A \Pi=\Pi A$
and~\eqref{e:t:lin:suf:1} to bound the first term~by
\begin{IEEEeqnarray*}{c}
\frac{x^\top \Pi^\top \bar M \Pi}{V(x_1,x_2)}
Ax
\le -\lambda 
\frac{x^\top \Pi^\top \bar M \Pi x}{V(x_1,x_2)}
=-\lambda V(x_1,x_2).
\end{IEEEeqnarray*}
Moreover,
$\Pi(D_{12} w_1+B_{12}u_3)=0$
as
$D_{12} w_1+B_{12}u_3\in R$. Then we use the Cauchy-Schwarz
inequality to bound
\begin{multline*}
\frac{x^\top \Pi^\top \bar M \Pi}{V(x_1,x_2)}
\bigg(
\begin{bmatrix}B_1\\B_2K_4 \end{bmatrix}
u_1
+
\begin{bmatrix}0\\D_2 \end{bmatrix}
\Delta w
\bigg)\le \\
|\sqrt{\bar M}\Pi\begin{bmatrix}B_1\\B_2K_4 \end{bmatrix}|
|u_1|
+
|\sqrt{\bar M}\Pi\begin{bmatrix}0\\D_2 \end{bmatrix}|
|w_2-w_1|
\end{multline*}
and we see that $V$ is a simulation function with the associated
comparison functions given by $\alpha= \id$ and for all $r\in\R_{\ge0}$ and
$s\in \R^p_{\ge0}$ by 
$\lambda(r)=\lambda r$, 
\begin{IEEEeqnarray*}{c}\label{e:gains}
\begin{IEEEeqnarraybox}[][c]{c,t,c}
\rho(r)=
|\sqrt{\bar M}\Pi\begin{bmatrix}B_1\\B_2K_4 \end{bmatrix}|r\;& and &
\mu(s)=\sum_{i=1}^p
|\sqrt{\bar M}\Pi\begin{bmatrix}0\\D_2 \end{bmatrix}|s_i.
\end{IEEEeqnarraybox}
\end{IEEEeqnarray*}
If $\im B_{21}\subseteq R+\im  B_{12}$, for every $u_1$  we choose
$u_2$ differently by $u_2=K_{12}(x_1;x_2)+u_3+u_4$ with
$u_{4}$ so that $B_{21} u_1+ B_{12}u_4\in R$ which implies
$\Pi(B_{21} u_1+B_{21}u_{4})=0$ and the term of the left-hand-side
of~\eqref{inequality1} associated with $u_1$ vanishes.
\end{proof}

Theorem~\ref{t:lin:suf} gives rise to the following definition.
\begin{definition}\label{d:srsf}
Let  $\Sigma_i=(A_i,B_i,C_i,D_i)$, $i\in\{1,2\}$ be 
two linear control systems with
the same internal input space dimension and the same output space dimension.
Let the matrices $A_{12}, B_{12},C_{12},D_{12}$ be given
by~\eqref{e:auxmat}.
We say that a relation
  $R\subseteq \R^{n_1}\times \R^{n_2}$
\emph{induces a simulation function from $\Sigma_1$ to $\Sigma_2$} if
it satisfies~\eqref{e:lin:nec:estab}-\eqref{e:lin:nec:ker}.
\end{definition}

Theorems~\ref{t:lin:nec} and~\ref{t:lin:suf} facilitate a direct
comparison of  simulation
functions with the notion of a 
simulation relation $R$ from $\Sigma_1$ to
$\Sigma_2$~\cite{vdS04}. A relation $R\subseteq
\R^{n_1}\times\R^{n_2}$ is a \emph{simulation relation
from $\Sigma_1$ to $\Sigma_2$} if for every $(x_1,x_2)\in R$,
$\nu_1$ and $\omega_1\equiv\omega_2$, there exists $\nu_2$ so that
\begin{IEEEeqnarray}{c}\label{e:SR}
\begin{IEEEeqnarraybox}[][c]{c'l}
\forall_{t\in\R_{\ge0}}:&(\xi_{1,x_1,\nu_1,\omega_1}(t),\xi_{2,x_2,\nu_2,\omega_2}(t))\in R\\
\forall_{t\in\R_{\ge0}}:&\zeta_{1,x_1,\nu_1,\omega_1}(t)=\zeta_{2,x_2,\nu_2,\omega_2}(t).
\end{IEEEeqnarraybox}
\end{IEEEeqnarray}
This notion of simulation relation was introduced in~\cite{vdS04} in
the context of \emph{verification} for linear systems with two types of
inputs. Due to the verification context, in~\cite{vdS04} the internal
input is interpreted as control input and the external inputs as
disturbances. While in our approach, we use the external input as
control input that we refine from $\Sigma_1$ to $\Sigma_2$ and the
internal input is used for the interconnection of the subsystems.
Nevertheless, mathematically, both notions are closely related and
the authors in~\cite{vdS04} characterized simulation relations from
$\Sigma_1$ to $\Sigma_2$ in terms of
conditions~\eqref{e:lin:nec:cinv}-\eqref{e:lin:nec:exactSR}. On one
hand, two systems $\Sigma_1$ and $\Sigma_2$ that are related via a simulation function (or equivalently a relation that induces a simulation
function) needs to satisfy~\eqref{e:lin:nec:exactSR} only if
$\rho\equiv 0$ should hold.
As a result, given an output trajectory
$\zeta_{1,x_1,\nu_1,\omega_1}(t)$ of $\Sigma_1$, there does not necessarily exist an output
trajectory 
$\zeta_{2,x_2,\nu_2,\omega_2}(t)$ of $\Sigma_2$ so that both trajectories are
identical. On the other
hand, a simulation relation $R$ is not required to be externally
stabilizable~\eqref{e:lin:nec:estab}. The external stabilizability
in the context of simulation functions
allows $\zeta_{1,x_1,\nu_1,\omega_1}(t)$ and
$\zeta_{2,x_2,\nu_2,\omega_2}(t)$ 
to be driven by the different internal inputs $\omega_1\not\equiv
\omega_2$ and  the initial states are not restricted to satisfy $(x_1,x_2)\in
R$. In view of~\eqref{inequality} the effect of the different internal
inputs
on the output difference is
bounded and the effect of the freely chosen initial states vanishes over time.

We conclude this subsection with the characterization of a relation inducing a simulation function from $\Sigma_1$
to $\Sigma_2$ (or from $\Sigma_2$ to $\Sigma_1$)  that is defined in
terms of a matrix $P\in \R^{n_2\times
n_1}$ by
\begin{IEEEeqnarray}{c}\label{e:lin:relP}
  R=\{(x_1;x_2)\in \R^{n_1}\times \R^{n_2}\mid Px_1 =x_2\}.
\end{IEEEeqnarray}
We use this result in the next subsection to construct an
approximate abstraction
$\hat \Sigma$ of a given linear control system~$\Sigma$.

\begin{theorem}\label{t:lin:construction}
Consider two linear control systems $\Sigma_i=(A_i,B_i,C_i,D_i)$, $i\in\{1,2\}$
with the same internal input space dimension and the same output space dimension.
Let $R$ be given by~\eqref{e:lin:relP}
with the matrix $P\in \R^{n_2\times
n_1}$. The relation $R$ induces a simulation function from $\Sigma_1$
to $\Sigma_2$
iff there exists matrices $K_1,K_2,K_3$ of appropriate dimensions
so that the following holds
\begin{IEEEeqnarray}{l}
\IEEEyesnumber
\IEEEyessubnumber
\label{e:lin:const:1:a}
A_2+B_2K_1\text{ is Hurwitz}\\
\IEEEyessubnumber
\label{e:lin:const:1:b}
A_2P= PA_1+B_2 K_2\\
\IEEEyessubnumber
\label{e:lin:const:1:c}
D_2= P D_1+B_2K_3\\
\IEEEyessubnumber
\label{e:lin:const:1:d}
C_1= C_2P.
\end{IEEEeqnarray}
Moreover, \eqref{e:lin:nec:exactSR} holds iff
there exists $K_4$ so that
\begin{IEEEeqnarray}{c}
\IEEEyessubnumber
\label{e:lin:const:1:e}
PB_1=B_2K_4.
\end{IEEEeqnarray}
\end{theorem}
\begin{proof}
First, we show that $R$
satisfies~\eqref{e:lin:nec:cinv}-\eqref{e:lin:nec:ker}
(\eqref{e:lin:nec:cinv}-\eqref{e:lin:nec:exactSR})
iff
\eqref{e:lin:const:1:b}-\eqref{e:lin:const:1:d}
(\eqref{e:lin:const:1:b}-\eqref{e:lin:const:1:e}) holds. By the
definition of $R$ it is straightforward to establish the equivalences
\eqref{e:lin:nec:cinv} $\iff$  \eqref{e:lin:const:1:b},
\eqref{e:lin:nec:D} $\iff$  \eqref{e:lin:const:1:c},
\eqref{e:lin:nec:ker} $\iff$  \eqref{e:lin:const:1:d} and
\eqref{e:lin:nec:exactSR} $\iff$
\eqref{e:lin:const:1:e}. 
Now we assume that $R$ is
$(A_{12},B_{12})$-controlled invariant. Let $K_{12}=[K'_1\, K_1]$ so that
$(A_{12}+B_{12}K_{12})R\subseteq R$. Then we pick $T$
in~\eqref{e:lin:suf:structure} by
\begin{IEEEeqnarray*}{c}
T=\begin{bmatrix}
    I_{n_1}& 0\\
    P&I_{n_2}
  \end{bmatrix}, \text{where } 
  \im\begin{bmatrix}
    I_{n_1}\\
    P
  \end{bmatrix}=R
\end{IEEEeqnarray*}
and observe that
$-PA_1+B_2K'_1+(A_2+B_2K_1)P=0$ and $F_{22}=A_2+B_2K_1$, which shows
that \eqref{e:lin:const:1:b} holds and, consequently,
\eqref{e:lin:nec:estab} holds iff
\eqref{e:lin:const:1:a} holds.
\end{proof}
The following corollary readily follows from the proofs of
Theorem~\ref{t:lin:suf} and~\ref{t:lin:construction}.
\begin{corollary}\label{c:lin:construction}
Suppose that \eqref{e:lin:const:1:a}-\eqref{e:lin:const:1:d} hold. Let
$M\in\R^{n_2\times n_2}$ be a symmetric positive definite matrix that
satisfies
\begin{IEEEeqnarray*}{l}
C_2^\top C_2\le M\\
(A_2+B_2K_1)^\top M+M(A_2+B_2K_1)\le -2\lambda M
\end{IEEEeqnarray*}
for some $\lambda\in\R_{>0}$. Then a simulation function from
  $\Sigma_1$ to $\Sigma_2$ is given by
\begin{IEEEeqnarray*}{c}
 V(x_1, x_2)=\big((x_2-Px_1)^\top M(x_2-Px_1)\big)^{\tfrac{1}{2}}
\end{IEEEeqnarray*}
and the interface function that maps $x_1$, $x_2$, $u_1$, $w_1$ to
$u_2$ so that~\eqref{inequality1} holds is given by
\begin{IEEEeqnarray*}{c}
u_2=K_1(x_2-Px_1)-K_2x_1-K_3w_1+K_4u_1,
\end{IEEEeqnarray*}
where $K_4$ is given implicitly as the matrix that minimizes
$|\sqrt{M}(PB_1-B_2K_4)|$.  Let $d_i\in\R^n$, $i\in\intcc{1;p}$
denote the columns of $D_2$.
The comparison functions associated with $V$ follow for all $r\in\R_{\ge0}$  and
 $s\in\R_{\ge0}^p$ by $\alpha(r)=r$,
$\lambda(r)=\lambda r$, 
\begin{IEEEeqnarray*}{l}
\rho(r)=|\sqrt{M} (PB_1-B_2K_4)|r,\\
\mu(s_1,\ldots,s_p)=|\sqrt{M}d_{1}|s_1+\ldots+|\sqrt{M}d_{p}|s_p.
\end{IEEEeqnarray*}
\end{corollary}

\subsection{Construction of Approximate Abstractions}
\label{ss:con:lin}

In this subsection, we are interested in the construction of an
approximate
abstraction $\hat \Sigma=(\hat A,\hat B,\hat C,\hat D)$ for a given linear control system $\Sigma=(A,B,C,D)$ together with a
square-root-of-quadratic simulation function from $\hat \Sigma$
to~$\Sigma$. Given the fact that any two asymptotically stable
  linear systems $\Sigma$
  and $\hat \Sigma$ (with suitable internal input and output space dimensions) can be related via a simulation function, we follow the
  approach in~\cite{GP09} to construct abstractions of linear control systems,
  and ask not only for a simulation function from
  $\hat\Sigma$ to $\Sigma$, but additionally require that there exists a simulation relation\footnote{Actually, the authors of~\cite{GP09}
show that $\hat\Sigma$ is $\hat P$-related to $\Sigma$
(see~\cite[Def.~3]{GP09}), which, when we omit the internal inputs, is
equivalent to $\hat R$ being a simulation relation from $\Sigma$ to
$\hat \Sigma$.} from $\Sigma$ to $\hat \Sigma$, which ensures that nice
properties like controllability of $\Sigma$ are preserved on the
abstraction $\hat \Sigma$.
The construction is based on the assumption that
\begin{IEEEeqnarray}{l}
\IEEEyesnumber
\label{e:lin:const}
\IEEEyessubnumber
\label{e:lin:const:a}
(A,B) \text{ is stabilizable},
\end{IEEEeqnarray}
and on the existence of a matrix $P \in \R^{n\times
\hat n}$ with a trivial kernel that satisfies
\begin{IEEEeqnarray}{l}
\IEEEyessubnumber
\label{e:lin:const:b}
A\im P\subseteq \im P+\im B\\
\IEEEyessubnumber
\label{e:lin:const:c}
\im D\subseteq \im P+\im B\\
\IEEEyessubnumber
\label{e:lin:const:d}
\im P+\ker C=\R^n.
\end{IEEEeqnarray}
In~\cite{GP09} conditions~\eqref{e:lin:const:a}, \eqref{e:lin:const:b}
and~\eqref{e:lin:const:d} were used to construct an abstraction $\hat
\Sigma$ and a square-root-of-quadratic simulation function $V$
from $\hat \Sigma$ to $\Sigma$ together with a simulation relation
$\hat R=\{(x;\hat x)\mid \hat Px=\hat x\}$ (for some $\hat
P\in\R^{\hat n\times n}$)
from $\Sigma$ to $\hat \Sigma$. 
In this paper, we extend the scheme
in~\cite{GP09} in the following directions. First, we add
condition~\eqref{e:lin:const:c} in order to be able to account for
systems with internal and external inputs. Second, we show that the
simulation relation $\hat R$ actually induces a simulation function
from $\Sigma$ to $\hat \Sigma$. Third, and most importantly, using the
novel geometric characterization of simulation functions,
we show that the conditions
\eqref{e:lin:const:a}-\eqref{e:lin:const:d} are not only \emph{sufficient}
but actually \emph{necessary} for the existence of an abstraction $\hat
\Sigma$ so that the relation $R=\{(\hat x,x)\mid P\hat x=x\}$ induces
a simulation function from $\hat \Sigma$ to $\Sigma$ and $\hat R$
induces a simulation function from $\Sigma$ to~$\hat \Sigma$.

\begin{theorem}\label{t:lin:Pconst}
Consider $\Sigma=(A,B,C,D)$ and
\begin{IEEEeqnarray*}{c}
R=\{(\hat x;x)\in \R^{\hat n}\times \R^n\mid P\hat x=x\}
\end{IEEEeqnarray*}
with $P\in\R^{n\times \hat n}$, $\ker P=0$.
There exist $\hat \Sigma=(\hat A,\hat B,\hat C,\hat
D)$ with the same internal input space dim.~and
the same output space dim.~as $\Sigma$ 
and 
\begin{IEEEeqnarraybox}{c}
\hat R=\{(x;\hat x)\in \R^{n}\times \R^{\hat n}\mid \hat P x=\hat x\}
\end{IEEEeqnarraybox}
with $\hat P\in\R^{\hat n\times n}$, 
 so that $R$ induces a simulation function from $\hat \Sigma$ to $\Sigma$
and $\hat R$ 
induces a simulation function from $\Sigma$ to $\hat\Sigma$ 
iff~\eqref{e:lin:const:a}-\eqref{e:lin:const:d} hold.
\end{theorem}

\begin{proof}
Let $R$ ($\hat R$) induces a simulation function from $\hat \Sigma$ to
$\Sigma$ ($\Sigma$ to $\hat \Sigma$). From
Theorem~\ref{t:lin:construction} it follows that~\eqref{e:lin:const:1:a}
implies~\eqref{e:lin:const:a}, \eqref{e:lin:const:1:b}
implies~\eqref{e:lin:const:b} and \eqref{e:lin:const:1:c}
implies~\eqref{e:lin:const:c}. From~\eqref{e:lin:const:1:d} it follows
that $\hat C =C P$ and $\hat C\hat P=C$, which implies that $C P\hat
P=C$. Since $\ker P=0$, Lemma~3 in~\cite{GP09} is applicable and we
obtain~\eqref{e:lin:const:d}.
Now suppose that~\eqref{e:lin:const:a}-\eqref{e:lin:const:d} hold.
Let $\hat C=CP$ and pick $\hat A$ and $\hat C$ together with $K_1,K_2,K_3$ 
so that \eqref{e:lin:const:1:a}-\eqref{e:lin:const:1:d} hold for $A$,
$B$, $C$, $D$, $\hat A$, $\hat B$, $\hat C$, $\hat D$ in place of $A_2$, 
$B_2$, $C_2$, $D_2$, $A_1$, $B_1$, $C_1$, $D_1$, respectively.
Theorem~\ref{t:lin:construction} shows that $R$ induces a simulation
function from $\hat \Sigma=(\hat A,\hat
B,\hat C,\hat D)$ to $\Sigma$ for any $\hat B$ of appropriate
dimension. We continue to show that $\hat R$ induces a simulation
function from $\Sigma$ to $\hat \Sigma$. Again we use
Lemma~3 in ~\cite{GP09} to pick $\hat P$ with $\im \hat P=\R^{\hat
n}$ so that $\hat C\hat P=C$, $\hat
PP=I_{\hat n}$ and $P\hat P+EF=I_n$ for some matrices $E$ and $F$ of
appropriate dimension with
$\im E=\ker C$. Let $\hat B=[\hat PB\; \hat PAE]$. We derive $\hat
A\hat P=\hat P P \hat A \hat P=\hat PA P\hat P- \hat PB(K_2+K_1P)\hat P=
\hat P A-\hat PAEF -\hat P B(K_2+K_1P)\hat P=\hat P A+\hat
 B[-(K_2+K_1P)^\top\;-F^\top]^\top$ and $\hat D=\hat P P \hat D
 =\hat P(D-BK_3)=\hat P D+\hat B[-K_3^\top\; 0]^\top$. Additionally, we
 have $\hat P B=\hat B[I_{m}\;0]^\top$ and it follows that
 $\hat R$ satisfies~\eqref{e:lin:const:1:b}-\eqref{e:lin:const:1:e} for
 $\hat P$, $\hat A$, $\hat B$, $\hat C$, $\hat D$, $A$,
$B$, $C$, $D$ in place of $P$, $A_2$, $B_2$, $C_2$, $D_2$, $A_1$,
$B_1$, $C_1$, $D_1$, respectively, which shows that $\hat{R}$ is a simulation
relation from $\Sigma$ to $\hat\Sigma$ \cite[Prp. 5.2]{vdS04}. Moreover, $\im \hat P=\R^{\hat
n}$. As $(A,B)$ is stabilizable we use~\eqref{e:SR} to verify that
$(\hat A,\hat B)$ is stabilizable as well. Hence, there exists a
matrix $\hat K_1$ so that~\eqref{e:lin:const:1:a} holds. It follows
that $\hat R$ induces a simulation function from $\Sigma$ to $\hat
\Sigma$.
\end{proof}

We summarize the construction of an approximate abstraction of a stabilizable
control system $\Sigma=(A,B,C,D)$ in Table~\ref{tb:1}.
\begin{table}[h]
\begin{mdframed}
\begin{enumerate}
  \item Compute $M$ and $K_1$ so that $ C^\top C\le M$ and \\
    \begin{IEEEeqnarraybox}{l}
      (A+BK_1)^\top M+M(A+BK_1)\le -2\lambda M
    \end{IEEEeqnarraybox}
    holds.
  \item Determine $P$ with $\ker P=0$  that
  satisfies~\eqref{e:lin:const:b}-\eqref{e:lin:const:d}\\ and $\hat P$
  so that $CP\hat P=C$ and $\im \hat P=\R^{\hat n}$.
  \item Determine $\hat A$ and $K_2$ so that $AP=P\hat A+BK_2$ holds.
  \item Determine $\hat D$ and $K_3$ so that $D=P\hat D+BK_3$ holds.
  \item The matrices $\hat B$ and $\hat C$ follow by  $\hat B=[\hat
  PB\; \hat PAE]$\\ where $\im
  E=\ker C$ and $\hat C=CP$.
\end{enumerate}
\end{mdframed}
\caption{Construction of an approximate abstraction $\hat \Sigma$.}\label{tb:1}
\vspace{-0.3cm}
\end{table}
The associated simulation function from $\hat \Sigma$ to $\Sigma$ follows from
Corollary~\ref{c:lin:construction} to
  \mbox{$V(\hat x, x) =\sqrt{(x-P\hat x)^\top M(x-P\hat x)}$}
and the interface function that maps $\hat x$, $x$, $\hat u$, $\hat w$ to
$u$ so that~\eqref{inequality1} holds is given by
\begin{IEEEeqnarraybox}{c}
u=K_1(x-P\hat x)-K_2\hat x-K_3\hat w+K_4\hat u.
\end{IEEEeqnarraybox}
The matrix $K_4$ is given as the one that minimizes
$|\sqrt{M}(P\hat B-BK_4)|$ which can be computed according to~\cite[Prp.~1]{GP09}.

Note that Theorem~\ref{t:lin:Pconst}
  provides only \emph{structural} conditions for the construction of
approximate abstractions of linear control systems and it is an interesting
open question on how to pick the different matrices outlined in Table~\ref{tb:1}
(within the allowed domains) so as to obtain approximate abstractions with
\emph{optimal} approximation accuracies.

\section{An Example}\label{example}

Let us consider the compositional construction of an approximate abstraction together with a
simulation function for an interconnected linear control system
illustrated in Figure~\ref{f:ex1}. We consider two triple integrators
($\Sigma_1$ and $\Sigma_3$) which are organized in a feedback connection, where 
the output of $\Sigma_3$ is directly connected to the input of $\Sigma_1$ and the
output of $\Sigma_1$ is connected to the input of $\Sigma_3$ via two
two-dimensional systems $\Sigma_2$ and $\Sigma_4$.
\begin{figure}[h]
\centering

  \begin{tikzpicture}[auto, node distance=2cm, >=latex']

  \tikzstyle{block}= [draw,
                      thick,
                      rectangle,
                      minimum height = 7mm, 
                      minimum width = 7mm]

    \node[block] (sys1) at (-.5,0) {$\Sigma_1$};
    \node[block] (sys2) at (2,-.5) {$\Sigma_2$};
    \node[block] (sys3) at (4.5,0) {$\Sigma_3$};
    \node[block] (sys4) at (2,.5) {$\Sigma_4$};

    \node (i1) at (1,-1) {};
    \node (i2) at (3,-1) {};

    \draw[thick,<-] ($(sys1.west)+(0,0.15)$) -- node[near end,above] {\footnotesize $u_{1}$} ++(-1.5,0);

    \draw[thick,->] ($(sys1.east)+(0,0)$) -- node[near end,above] {\footnotesize $y_{11}$} ++(1.5,0);
    \draw[thick,->] ($(sys3.east)+(0,0.15)$) -- node[near end] {\footnotesize $y_{33}$} ++(1.5,0);
    \draw[thick,->] ($(sys4.east)+(0,.15)$)  -- node[near end] {\footnotesize $y_{44}$} ++(1.5,0);
    \draw[thick,->] ($(sys2.east)-(0,.15)$)  -- node[near end,below] {\footnotesize $y_{22}$} ++(1.5,0);

    \draw[thick,<-] ($(sys3.west)+(0,0)$) -- node[very near end,above] {\footnotesize $u_{3}$} ++(-1.5,0);

    \draw[->] ($(sys1.east)+(0,.15)$) -|  (.5,.15) |- node[near end, above] {\footnotesize $y_{14}$} (sys4.west);
    \draw[->] ($(sys1.east)+(0,-.15)$) -| (.5,-.15) |- node[near end, below] {\footnotesize $y_{12}$} (sys2.west);

    \draw[->] ($(sys4.east)-(0,.15)$) -|  (3.5,.15) |- node[above,near end] {\footnotesize $y_{43}$} ($(sys3.west)+(0,.15)$);
    \draw[->] ($(sys2.east)+(0,.15)$) -|  (3.5,-.15) |- node[below,near end] {\footnotesize $y_{23}$} ($(sys3.west)-(0,.15)$);

    \draw[->] 
    ($(sys3.east)-(0,0.15)$) -| node[near start,below] {\footnotesize $y_{31}$}
    ++(0.5,-1) -- 
    ++(-7,0) |-
    ($(sys1.west)-(0,0.15)$);

  \end{tikzpicture}
  \caption{The interconnected system
  $\mathcal{I}(\Sigma_1,\Sigma_2,\Sigma_3,\Sigma_4)$.}\label{f:ex1}
  \vspace{-0.2cm}
\end{figure}
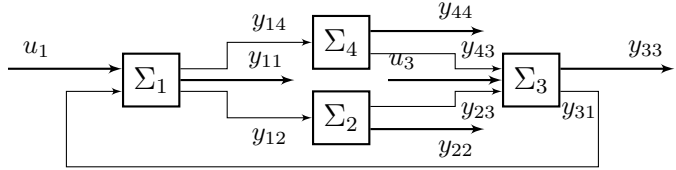
The system matrices are accordingly set to
\begin{IEEEeqnarray*}{c'c}
A_1=A_3
=
\begin{bmatrix}
 0 & 1 & 0\\
 0 & 0 & 1\\
 0 & 0  & 0
\end{bmatrix}
,&
\begin{IEEEeqnarraybox}[][c]{l}
B^\top_1=B^\top_3=
\begin{bmatrix} 0& 0& 1 \end{bmatrix},\\
C_{11}= C_{33}
=
\begin{bmatrix} 1& 0& 0 \end{bmatrix},
\end{IEEEeqnarraybox}
\end{IEEEeqnarray*}
and
\begin{IEEEeqnarray*}{l}
A_2=A_4
=
\begin{bmatrix}
 0 & 1 \\
 -6 & -5
\end{bmatrix}
,\;
B_2^\top=
B_4^\top=
C_{22}=C_{44}=
\begin{bmatrix} 1& 0 \end{bmatrix}.
\end{IEEEeqnarray*}
Whereas the interconnection matrices $C_{ij}$, $D_{ij}$ are given by
\begin{align*}
C_{14}&=C_{12}=C_{31}=
\begin{bmatrix} 1& 0& 0 \end{bmatrix},
C_{23}=C_{43}=
\begin{bmatrix} 1& 0 \end{bmatrix}\\
D_{13}&=
\begin{bmatrix}
 0\\
 0\\
 d_{1}
\end{bmatrix}
,\;
D_{21}=D_{41}=
\begin{bmatrix}
 -d_2\\
 \phantom{-}2d_{2}
\end{bmatrix}
,\;
D_{34}=D_{32}=
\begin{bmatrix}
 0\\
 0\\
 d_3
\end{bmatrix},
\end{align*}
for some $d_i\in\R$. The remaining $C_{ij}$ and $D_{ij}$ are given by
  zero matrices. We summarize the internal input and
output matrices by
\begin{IEEEeqnarray*}{c'c}
C_{1}=C_3=
\begin{bmatrix} 1& 0& 0 \end{bmatrix},&
C_{2}=C_{4}=
\begin{bmatrix} 1& 0 \end{bmatrix},\\
D_{1}=
\begin{bmatrix}
 0\\
 0\\
 d_{1}
\end{bmatrix}
,\;
D_{2}=D_{4}=
\begin{bmatrix}
 -d_2\\
 \phantom{-}2d_{2}
\end{bmatrix}
,&
D_{3}=
\begin{bmatrix}
 0 & 0\\
 0 & 0\\
 d_{3} &d_{3}
\end{bmatrix}.
\end{IEEEeqnarray*}

{\bf The Abstract System.}
We continue the example by applying the procedure outlined in
Table~\ref{tb:1}
to construct an abstraction $\hat \Sigma_i$ of each subsystem~$\Sigma_i$.

We start by computing $M_i$, $K_{i,1}$ and $\lambda_i$, for
$i\in\{1,3\}$,
such that the matrix inequalities in 1) of Table~\ref{tb:1}
hold. To this end, we solve the linear matrix inequality given by
equations (6) and (7) in~\cite{GP09}.
We obtain
\begin{IEEEeqnarray*}{c}
M_i=
\begin{bmatrix}
 4.59 &    4.07 &    0.90\\
 4.07 &    4.72 &    1.24\\
 0.90 &    1.24 &    0.61
\end{bmatrix}
,\;
K_{i,1}=
-\begin{bmatrix}
5.13 &7.12 & 3.03
\end{bmatrix},
\end{IEEEeqnarray*}
with $\lambda_i=1$.
Next we determine $P_i$ for $\Sigma_i$ so
that~\eqref{e:lin:const:b}-\eqref{e:lin:const:d} hold by
$P_i= [ 1\; 0 \;0 ]^\top$. Following 2) through 5) in Table~\ref{tb:1} we obtain $\hat \Sigma_i$ by
\begin{IEEEeqnarray*}{c}
\hat A_i=0,\;\hat B_i=1,\;\hat D_i=0,\;\hat C_i=1,
\end{IEEEeqnarray*}
together with the matrices for the interface $K_{i,2}=0$,
$K_{1,3}=d_1$, $K_{3,3}=[d_3\; d_3]$ and $K_{i,4}=1.47$.
The simulation functions follow by
$V_i(\hat x_i,x_i)=\sqrt{(x_i-P_i\hat x_i)^\top M_i(x_i-P_i\hat x_i)}$
and the associated comparison functions by $\alpha_1=\alpha_3=\id$ and 
\begin{IEEEeqnarray*}{c,c,l}
\lambda_1(r)=r,&\rho_{1}(r)=1.81r,&
\mu_{1}(r_2,r_3,r_4)=0.78d_{1}r_3\\
 \lambda_3(r)=r,&\rho_{3}(r)=1.81r,&
\mu_{3}(r_1,r_2,r_4)=0.78d_{3}(r_2+r_4).
\end{IEEEeqnarray*}

We continue with subsystems $\Sigma_2$ and $\Sigma_4$.
Since the subsystems $\Sigma_2$ and
$\Sigma_4$ have no external inputs, it is necessary that the
matrices $A_2$ and $A_4$ are Hurwitz in order to be able to find
matrices $M_2$ and $M_4$ that satisfy the matrix inequalities in 1) of
Table~\ref{tb:1}. This holds for our example and we 
compute 
\begin{IEEEeqnarray*}{c'c}
M_2=M_4
=
\begin{bmatrix}
    26 & 10\\
    10 & 4
\end{bmatrix},&K_{2,1}=K_{4,1}=0
\end{IEEEeqnarray*}
with $\lambda_2=\lambda_4=2$.
Also the conditions~\eqref{e:lin:const:b} and~\eqref{e:lin:const:c}
simplify in the absence of
any external inputs. It follows that $\im P_i$
needs to be an $A_i$-invariant subspace that contains
$\im D_i$. In this case, we can use the
Algorithm~3.2.1 in~\cite{BM92}
to compute the minimal $A_i$-invariant subspace that contains $\im
D_i$. We obtain $P_i=[1\;-2]^\top$, $i\in\{2,4\}$, and the
abstractions $\hat\Sigma_i$ by 
\begin{IEEEeqnarray*}{c}
  \hat A_i=-2$, $\hat B_i=1$, $\hat D_i=-d_2$, $\hat C_i=1.
\end{IEEEeqnarray*}
As before we obtain the square-root-of-quadratic simulation function,
defined by $P_i$ and $M_i$.
The associated interface follows by $k_i\equiv 0$.
The comparison functions associated with the simulation function $V_i$
are given by $\alpha_i=\id$, $\lambda_i(r)=2r$, $\rho_{i}(r)=1.41r$ and
$\mu_{i}(r_1,r_2,r_3)=1.41d_{2}r_1$.

{\bf The Composition.}
We apply Theorem~\ref{t:ic} to obtain 
a simulation function from $\mathcal{I}(\hat \Sigma_1,\hat \Sigma_2,\hat
\Sigma_3,\hat \Sigma_4)$ to
$\mathcal{I}(\Sigma_1,\Sigma_2,\Sigma_3,\Sigma_4)$. The functions 
$\Lambda$ and $\Gamma$ are linear and identified with 
\begin{IEEEeqnarray*}{c,c}
\Lambda
=
\begin{bmatrix}
1 & 0 & 0 & 0\\
0& 2  & 0 & 0\\
0 & 0 &1 & 0\\
0 & 0 & 0 & 2
\end{bmatrix},&
\Gamma
=
\begin{bmatrix}
0 & 0 & 0.78d_1 & 0\\
1.41d_2 & 0 & 0 & 0\\
0 & 0.78d_3 & 0 & 0.78d_3\\
1.41d_2 & 0 & 0 & 0\\
\end{bmatrix}.
\end{IEEEeqnarray*}
In order to be able to apply Theorem~\ref{t:ic}, we need to assure that
the spectral radius of
$\Gamma\Lambda^{-1}$ is strictly less than one so that 
there exists a vector $\eta\in\R^4_{>0}$ such that
$(1+\varepsilon)\Gamma\Lambda^{-1}\eta<\eta$ holds for some $\varepsilon>0$. We pick
$d_{1}=d_{2}=d_3=0.5$ and obtain $\lambda_{\max}(\Gamma\Lambda^{-1})=0.19$. We pick
$\eta=\begin{bmatrix} 0.4 & 0.6 &  0.5  &  0.6\end{bmatrix}^\top$
and verify that $(1+\varepsilon)\Gamma\Lambda^{-1}\eta<\eta$ holds for
$\varepsilon=4$. Certainly,
$\lambda_i/\eta_i r$ is differentiable and satisfies~\eqref{e:sigmalambda}.
We apply Theorem~\ref{t:ic} and obtain
$V(\hat x, x)=\max_i\tfrac{\lambda_i}{\eta_i}V_i(\hat x_i, x_i)$
as  simulation function from $\mathcal{I}(\hat \Sigma_1,\hat \Sigma_2,\hat
\Sigma_3,\hat \Sigma_4)$ to
$\mathcal{I}(\Sigma_1,\Sigma_2,\Sigma_3,\Sigma_4)$, with the associated comparison functions given by
$\alpha(r)=r$, $\lambda(r)=4/5r$ and $\rho(r)=4.8r$, see
Remark~\ref{e:comp:bound:lin}. Hence, we obtain the bound
\begin{IEEEeqnarray}{c}\label{e:ex:bound}
|\hat \zeta(t)-\zeta(t)|
\le 
V(\hat \xi(t),\xi(t))
\le
\mathrm{e}^{-4/5t}V(\hat x, x)+5.9||\hat \nu||_\infty.
\IEEEeqnarraynumspace
\end{IEEEeqnarray}
Let
$V_{\mathrm{vec}}(t)=(V_1(\hat\xi_1(t),\xi_1(t));\ldots;V_4(\hat\xi_4(t),\xi_4(t)))$
and 
$\hat Z=(\rho_1(||\hat \nu_1||_\infty);\ldots;\rho_4(||\hat\nu_4||_\infty))$, then similarly
to~\eqref{e:ex:bound},
we get
\begin{IEEEeqnarray}{c}\label{e:ex:bound:2}
V_{\mathrm{vec}}(t)
\le
\mathrm{e}^{-\Lambda t}V_{\mathrm{vec}}(0)
+
\Gamma\Lambda^{-1}V_{\mathrm{vec}}(t)
+
\Lambda^{-1}\hat Z
\end{IEEEeqnarray}
which provides the bound $|\hat\zeta(t)-\zeta(t)|\le
|V_{\mathrm{vec}}(t)|$.

{\bf Controller Synthesis.} Let us now synthesize a
controller for 
$\Sigma$ via the abstraction
$\hat\Sigma$ to enforce the specification, defined by the LTL
formula \cite{BK08}
\begin{IEEEeqnarray}{c'c}\label{e:ex:spec}
\G \mathrm{S}\bigwedge_{i\in\intcc{1;3}}\G \F T_i, & S,T_i\subseteq \R^4,
\end{IEEEeqnarray}
which requires that any output trajectory $\zeta$ of the closed loop system
evolves inside the set $S$  and visits each $T_i$, $i\in\intcc{1;3}$
infinitely often, i.e., for all $t\in\R_{\ge0}$ $\zeta(t)\in S$ and for each
$i\in\intcc{1;3}$ there exists $t'\ge t$ so that $\zeta(t')\in T_i$,
see~\cite{BK08}. The specification is illustrated in Figure~\ref{f:results}.
We use {\tt SCOTS}~\cite{RZ16} to  synthesize a controller for
$\hat \Sigma$ to enforce~\eqref{e:ex:spec}. In the synthesis process we
restricted the abstract inputs to $\hat u_1,\hat u_3\in\intcc{-0.1,0.1}$ and $\hat
u_2=\hat u_4=0$ for all times. Given that we can set the initial states of
$\Sigma$ to $x_i=P_i\hat x_i$, so that $V(\hat x,x)=0$, we obtain a bound
from~\eqref{e:ex:bound}
on the output difference by $|\zeta(t)-\hat \zeta(t)|\le
V(\hat\xi(t),\xi(t))\le\bar V:=0.85$ for all $t\ge0$. An improved bound is obtained
from~\eqref{e:ex:bound:2} by noting that 
$V^{k+1}_{\mathrm{vec}}=\Gamma\Lambda V^{k}_{\mathrm{vec}}+\Lambda^{-1}\hat Z$ with
$V_{\mathrm{vec}}^0=(\eta_1/\lambda_1 \bar V;\ldots;\eta_4/\lambda_4 \bar V)$
provides an upper bound $|\zeta(t)-\hat
\zeta(t)|\le V_{\mathrm{vec}}(t)\le V^{k+1}_{\mathrm{vec}}$ for any $k\ge0$.

A closed loop trajectory of $\Sigma$ and $\hat \Sigma$ as well as the output
difference and the theoretical bound $V_{\mathrm{vec}}^\infty=\lim_{k\to\infty} V^{k}_{\mathrm{vec}}$ are illustrated in Figure~\ref{f:results}.
A bound for $||\nu_1||_\infty$ follows by
$|K_{1,1}(x_1-P_1\hat x_1)|+|K_{1,3}\hat w_1|+K_{1,4}|\hat u_1|\le 5.7$ where we used
$|x_1-P_1\hat x_1|\le V_1(x_1,\hat x_1)/\sqrt{\lambda_{\min}(M_1)}\le0.47$ and
$|\hat w_1|=|\hat y_3 |\le 6$. Similarly we obtain $||\hat \nu_3||_\infty\le
4.4$.  For the example
trajectory in Figure~\ref{f:results} the inputs $\nu_1$ and $\nu_3$ never 
 exceeded $1.2$ and $0.31$, respectively.
\begin{figure}[ht]
\begin{center}
\begin{tikzpicture}[>=latex']
\node at (0,0)         {\includegraphics[width=3.9cm]{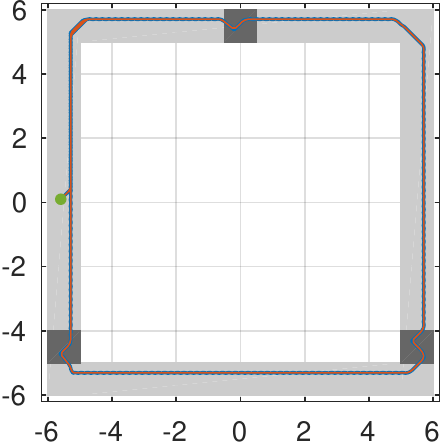}};
\node at (4.25cm,0)     {\includegraphics[width=3.9cm]{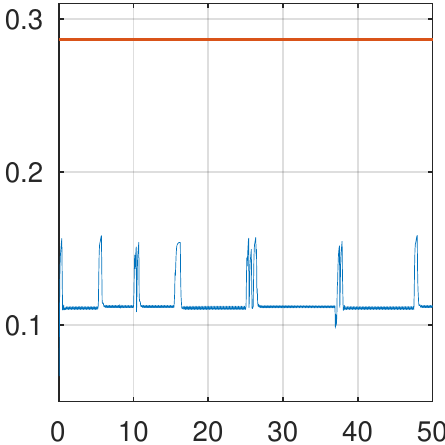}};

\node at (0.1cm,1.2cm) {\small $T_1$};
\node at (-1cm,-1cm) {\small $T_2$};
\node at (1.3cm,-1cm) {\small $T_3$};

\node at (0.25cm,-2.1cm) {\footnotesize $y_1$};
\node at (-2.2cm,0cm) {\footnotesize $y_3$};
\node at (4.5cm,-2.2cm) {\footnotesize $t[20sec]$};
\node at (5.5cm,1.4cm) {\footnotesize $|V^\infty_\mathrm{vec}|$};
\node at (5.3cm,-1.1cm) {\footnotesize $|\hat\zeta(t)-\zeta(t)|$};

\end{tikzpicture}
\end{center}
\vspace{-.5cm}
\caption{
  Left: The specification with closed loop trajectories of $\Sigma$ (red) and
$\hat \Sigma$ (blue). The green dot marks the initial state. The sets $S$ and
$T_i$ are given by
$S=\hat S\smallsetminus \check S$ with $\hat
S=\intcc{-6,6}\times\intcc{-1,1}\times\intcc{-6,6}\times\intcc{-1,1}$ and
$\check S=\intcc{-5,5}\times\intcc{-1,1}\times\intcc{-5,5}\times\intcc{-1,1}$, 
$T_1=\tfrac{1}{2}\intcc{-1,1}\times\intcc{-1,1}\times\intcc{5,6}\times\intcc{-1,1}$,
$T_2=\intcc{-6,-5}\times\intcc{-1,1}\times\intcc{-5,-4}\times\intcc{-1,1}$,
and $T_3=\intcc{5,6}\times\intcc{-1,1}\times\intcc{-5,-4}\times\intcc{-1,1}$.
Right: The output
difference (blue) and the upper bound obtained from~\eqref{e:ex:bound} (red).
} 
\label{f:results}
\end{figure}

\begin{remark}
As the controller synthesis algorithms implemented in {\tt SCOTS} operate
on a \emph{finite} abstraction of the concrete system, which is obtained by a
uniform discretization of the state space, it would not have been possible to
synthesize a controller for the original system $\Sigma$, 
without the
lower dimensional intermediate approximation $\hat\Sigma$. 
\end{remark}

\section{Summary}

In this paper we presented a compositional reasoning approach based on
a small gain type argument in connection with approximate abstractions
of nonlinear control systems. Given that the small gain type
condition is satisfied, we showed how to
construct an approximate abstraction together with a simulation
function for an interconnected nonlinear control system from 
the abstractions and simulation functions of its subsystems. Moreover,
for the special case of
linear control systems, we
characterized simulation functions in terms of a controlled invariant,
externally stabilizable subspace. Based on this characterization, we
proposed a particular scheme to construct approximate abstractions
together with the associate simulation functions. 

\bibliographystyle{plain}
\bibliography{../refs}

\appendix
\begin{lemma}\label{l:chainrule}
Let  $\alpha:\R_{\ge0}\to \R_{\ge0}$ be
a monotonically increasing function,
differentiable on $\R_{>0}$, and consider a 
function $f:\R^n\to \R_{\ge0}$. Then we have for all $x,v\in\R^n$
with $f(x)>0$
\begin{IEEEeqnarray}{c}\label{e:chainrule}
D^+(\alpha \circ f)(x,v)\le
\alpha'(f(x))D^+f(x,v).
\end{IEEEeqnarray}
\end{lemma}
\begin{proof}
As $\alpha$ is monotonically increasing and differentiable on $\R_{>0}$ we have for all $y_0\in\R_{>0}$
\begin{IEEEeqnarray*}{rCl}
0\le \alpha'(y_0)
&=&
\liminf_{y\to y_0, y<y_0} (\alpha (y)-\alpha(y_0))/(y-y_0)\\
&=&
\limsup_{y\to y_0, y>y_0} (\alpha (y)-\alpha(y_0))/(y-y_0).
\end{IEEEeqnarray*}
Let $x\in \R^n$ with $y_0=f(x)>0$ and $v\in \R^n$.
There exists a sequence $(t_i)_{i\in\N}$ in $\R_{>0}$ with
limit $0$ so that 
\begin{IEEEeqnarray*}{c}
D^+ (\alpha \circ f) (x, v) = \lim_{i\to \infty}\tfrac{1}{t_{i}}\big(\alpha(f(x+t_iv))-\alpha(f(x))\big).
\IEEEeqnarraynumspace
\end{IEEEeqnarray*}
If $f(x+t_iv)=f(x)$ for all $i\ge j$ for some $j\in\N$, we have  $D^+
(\alpha \circ f) (x, v)=0$ and $D^+ f (x, v)\ge 0$, which
shows~\eqref{e:chainrule}. If for every $j\in\N$ there
exists $i\ge j$ so that $f(x+t_iv)-f(x)> 0$ holds, we set $y_i=f(x+t_iv)$, $y=f(x)$ and pick a subsequence
$(t_{i_j})$ of $(t_i)$ so that $y_{i_j}>y$ for all $i_j$.
Since $(\alpha(y_{i_j})-\alpha(y))/(y_{i_j}-y)\ge0$ and $f(x+t_{i_j}v)-f(x) >0$,
for all $j\in\N$
we get 
\begin{IEEEeqnarray*}{l}
\lim_{j\to
\infty}\frac{\alpha(y_{i_j})-\alpha(y)}{y_{i_j}-y}\frac{f(x+t_{i_j}v)-f(x)}{t_{i_j}}\\
\le
\limsup_{j\to \infty} \frac{\alpha(y_{i_j})-\alpha(y)}{y_{i_j}-y}
\limsup_{j\to \infty} \frac{f(x+t_{i_j}v)-f(x)}{t_{i_j}}\\
\le
\alpha'(f(x)) D^+f(x,v).
\end{IEEEeqnarray*}
If $(y_i-y)_{i\in \N}$ contains infinitely negative
entries, we pick a subsequence $(t_{i_j})$  of $(t_i)$ so that
we have $y_{i_j}<y$
for all $i_j$ and use a similar reasoning as in the previous case to arrive
  at~\eqref{e:chainrule}. 
\end{proof}

{\it Proof of Theorem~\ref{theorem1}.}
Let us define the $\mathcal{K}_\infty$ function $\bar
\mu(s):=\mu(s,\ldots,s)$.
We consider the trajectories 
$(\xi,\zeta,\nu,\omega)$
and 
$(\hat \xi,\hat \zeta,\hat\nu,\hat \omega)$ 
of the control systems 
$\Sigma$ and $\hat\Sigma$,
respectively. We assume that $\nu$ is given such
that~\eqref{inequality1} holds with $x=\xi(t)$, $\hat x=\hat\xi(t)$,
$u=\nu(t)$, $\hat u=\hat \nu(t)$, $w=\omega(t)$, $\hat w=\hat
\omega(t)$ for all $t\in\R_{\ge0}$.
We define
$c=\lambda^{-1} \big( 2\rho(||\hat \nu||_\infty) +
2\bar\mu(||\omega-\hat \omega||_\infty) \big) $ 
and the set
\begin{IEEEeqnarraybox}{c}
S=\{(x;\hat x)\in\R^n\times\R^{\hat n}\mid V(\hat x, x)\le c\}.
\end{IEEEeqnarraybox}
From~\eqref{inequality1}, we see that 
$y(t):=V(\hat \xi(t),
\xi(t))$ satisfies, whenever $(\hat \xi(t),\xi(t))$ is outside the set $S$, i.e.
$y(t)>c$, the inequality
\begin{IEEEeqnarray}{c}\label{e:t:sf:help}
\begin{IEEEeqnarraybox}[][c]{rCl}
\mathsf{D}^{+}y(t,1)&=&
\mathsf{D}^{+} V\left((\hat \xi(t), \xi(t)),
\begin{bmatrix}
\hat f(\hat \xi(t),\hat \nu(t),\hat \omega(t))\\ 
f(\xi(t), \nu(t), \omega(t))
\end{bmatrix}
\right)\\
&\le&
-\tfrac{1}{2}\lambda(V(\hat \xi(t), \xi(t))),
\end{IEEEeqnarraybox}
\IEEEeqnarraynumspace
\end{IEEEeqnarray}
where the equality in~\eqref{e:t:sf:help} follows from~\cite[Thm~4.3, Rmk~4.4,
pp.~353]{RoucheHabetsLaloy77}. Hence, $y$ is decreasing for $y(t)>c$.
Suppose for all $t\in \intoo{a,b}\subseteq \R_{\ge0}$ we have $y(t)>c$, then
$t',t\in\intoo{a,b}$ with \mbox{$t'\le t$} implies 
$y(t') \le y(t) -\tfrac{1}{2}\int_t^{t'} y(s)\mathrm{d}s$ \cite[Thm~2.3,
Rmk~2.5]{RoucheHabetsLaloy77}.
We show that $S$ is forward invariant, i.e., if there
exists $t_0\ge0$ with $(\xi(t_0),\hat \xi(t_0))\in S$ then we have
\mbox{$(\xi(t),\hat \xi(t))\in S$} for all $t\ge t_0$.
Let $(\xi(t_0),\hat \xi(t_0))\in S$ and suppose to the contrary that
the trajectories leave $S$. Since $S$ is closed, there exists $t_1>
t_0$ and $\varepsilon\in\R_{>0}$ 
such that 
$y(t_1)\ge c+\varepsilon$. Let $t_1$ be
minimal for this choice of $\varepsilon$. Since 
$y(t)$ is continuous in $t$, there exists
$\delta>0$ with $\delta<t_1-t_0$, so that 
$y(t)>c$
holds for all $t\in t_1+\intoo{-\delta,\delta}$. However,
$y$ is decreasing on $\intoo{-\delta,\delta}$
which contradicts the minimality
of $t_1$. It follows that $S$ is forward invariant and the output
trajectories satisfy for all $t\ge t_0$ the inequality
\begin{IEEEeqnarray}{c}\label{e:t:sf:help1}
\begin{IEEEeqnarraybox}[][b]{rCl}
|\zeta(t)-\hat\zeta(t)|
&\le&
\alpha^{-1}(V(\hat\xi(t),\xi(t)))\\
&\le&
\alpha^{-1}\left(\lambda^{-1} \left(
2\rho(||\hat \nu||_\infty) +
2\bar\mu(||\omega-\hat \omega||_\infty)\right) \right)\\
&\le&
\gamma_{\mathrm{ext}}(||\hat \nu||_\infty) +
\gamma_{\mathrm{int}}(||\omega-\hat \omega||_\infty)
\end{IEEEeqnarraybox}
\end{IEEEeqnarray}
with the $\mathcal{K}\cup\{0\}$ functions
$\gamma_{\mathrm{ext}}(s):=\alpha^{-1}(\lambda^{-1}(4\rho(s)))$ and
 $\gamma_{\mathrm{int}}(s):=\alpha^{-1}(\lambda^{-1}(4\bar\mu(s)))$. Note that here we used the fact that
for any $\mathcal{K}\cup\{0\}$ function $\gamma$ the inequality $\gamma(a+b)\le
  \gamma(2a)+\gamma(2b)$ holds for all $a,b\in\R_{\ge0}$. 

We proceed with the analysis of the trajectories outside of~$S$. We
define $t_0=\inf\{t\mid (\xi(t),\hat \xi(t))\in S\}$ (possibly infinite) and observe that
the function $y(t)=V(\hat \xi(t), \xi(t))$ is absolutely continuous, 
since $V$ is locally Lipschitz and the
state trajectories are absolutely continuous. Hence, $y(t)$
is differentiable almost everywhere and
$y$ satisfies 
\begin{IEEEeqnarraybox}{c'c}
\dot y(t)\le-\tfrac{1}{2}\lambda(y(t))
\end{IEEEeqnarraybox}%
for almost all $t\in\intco{0,t_0}$.
  
Then we apply Lemma 4.4 in~\cite{LSW96} and
obtain a $\mathcal{KL}$ function $\bar\beta$ with $\bar\beta(r,0)=r$, depending only on
$\lambda$, so that $y(t)\le
\bar\beta(y(0),t)$ holds for all $t\in\intco{0,t_0}$. It follows
that the output trajectories satisfy for all $t\in\intco{0,t_0}$ the
inequality
\begin{IEEEeqnarray}{rCl}\label{e:t:sf:help2}
|\zeta(t)-\hat\zeta(t)|
&\le&
\beta(V(\hat \xi(0),\xi(0)),t)
\end{IEEEeqnarray}
with $\beta(r,t)=\alpha^{-1}(\bar\beta(r,t))$. By combining the
bounds~\eqref{e:t:sf:help1} and~\eqref{e:t:sf:help2} we obtain the
desired estimate~\eqref{inequality}.

{\it Proof of  Corollary~\ref{c:2}.}
It follows immediately by the previous derivations that $V$
satisfies~\eqref{e:sf} with the $\mathcal{KL}$ function given by
$\bar\beta$ (as determined in the previous proof) and the
$\mathcal{K}\cup\{0\}$ functions are given by
$\gamma_{\mathrm{ext}}(s):=\lambda^{-1}(4\rho(s))$ and
 $\gamma_{\mathrm{int}}(s):=\lambda^{-1}(4\bar\mu(s))$.

\end{document}